\documentclass[12pt,reqno]{amsart}

\usepackage{xcolor}
\usepackage{amsmath, amsthm, amssymb, amsxtra}
\usepackage{amsfonts}
\usepackage[utf8]{inputenc}
\usepackage[dvips]{epsfig}
\usepackage{graphicx}
\usepackage[english]{babel}
\usepackage{hyperref}

\usepackage{graphicx}
\usepackage{latexsym}
\usepackage{mathtools}
\usepackage{esint}
\usepackage{float}

\theoremstyle{plain}
\newtheorem{thm}{Theorem}[section]
\newtheorem{cor}[thm]{Corollary}
\newtheorem{lem}[thm]{Lemma}
\newtheorem{prop}[thm]{Proposition}

\theoremstyle{definition}
\newtheorem{defn}[thm]{Definition}

\theoremstyle{remark}

\numberwithin{equation}{section}

\newcommand{\average}{{\mathchoice {\kern1ex\vcenter{\hrule height.4pt
            width 6pt depth0pt} \kern-9.7pt} {\kern1ex\vcenter{\hrule
            height.4pt width 4.3pt depth0pt} \kern-7pt} {} {} }}

\newcommand{\R}{\mathbb R}
\newcommand{\N}{\mathbb N}

\renewcommand{\L}{\mathcal L}

\newcommand{\p}{\partial}

\newcommand{\comment}[1]{}

\begin{document}

\title[Extinction rates for nonradial solutions to the Stefan problem]{Extinction rates for nonradial solutions\\ to the Stefan problem}

\author[G. Fioravanti]{Gabriele Fioravanti}
\address{Universita degli Studi di Torino\newline
Dipartimento di Matematica “Giuseppe Peano”\newline
Via Carlo Alberto 10, 10124, Torino, Italy.}
\email{\tt gabriele.fioravanti@unito.it}

\author[X. Ros-Oton]{Xavier Ros-Oton}
\address{ICREA\newline
Pg.\ Llu\'is Companys 23, 08010 Barcelona, Spain\newline\indent
Universitat de Barcelona\newline
Departament de Matem\`atiques i Inform\`atica\newline
Gran Via de les Corts Catalanes 585, 08007 Barcelona, Spain \newline \indent
Centre de Recerca Matem\`atica \newline Barcelona, Spain.}
\email{\tt xros@icrea.cat}

\author[C. Torres-Latorre]{Clara Torres-Latorre}
\address{Instituto de Ciencias Matemáticas\newline Consejo Superior de Investigaciones Científicas\newline
C/ Nicolás Cabrera, 13-15, 28049 Madrid, Spain.}
\email{\tt clara.torres@icmat.es}

\begin{abstract}
We consider the one-phase Stefan problem describing the evolution of melting ice.
On the one hand, we focus on understanding the evolution of the free boundary near isolated singular points, and we establish for the first time upper and (more surprisingly) lower estimates for its evolution.
In 2D, these bounds almost match the best known ones for radial solutions, but hold for all solutions to the Stefan problem, with no extra assumption on the initial or boundary data.

On the other hand, as a consequence of our results, we also characterize the global regularity of the free boundary, as follows: it can be written as a graph $\{t=\Gamma(x)\}$, where $\Gamma$ is $C^1$ (and not $C^2$) near any singular points in the lower strata $\Sigma_m$, $m\leq n-2$.
Moreover, $\Gamma$ is not $C^1$ at singular points in $\Sigma_{n-1}$.
\end{abstract}




\maketitle

\section{Introduction}

The Stefan problem is probably the most classical and well-known free boundary problem \cite{LC31,Ste91}. 
It describes phase transitions, such as ice melting into water, and has been widely studied in the last 50 years \cite{ACS96,Caf77,CF78,CK10,Fri68,HS15,HR19,Kim03,KN78,Koc98,PSS07,Wei99}.

After the transformation $u(x,t)=\int_0^t \theta$, where $\theta\geq0$ is the temperature function, the one-phase Stefan problem becomes
\begin{equation}\label{eq:obstacleProblem}
    u_t - \Delta u = -\chi_{\{u>0\}},\quad u\geq 0,\quad u_t\geq 0,
\end{equation}
see \cite{Duv73,Fig18} for more details.
The moving interphase that separates the solid and liquid regions, $\partial\{u>0\}$, is often called free boundary.

The best known general results for the structure and regularity of such interphase may be summarized as follows:
\begin{itemize}
\item The free boundary splits into \textit{regular} points and \textit{singular} points.

\item The free boundary is $C^\infty$ near any regular point \cite{Caf77,KN77,KN78}.

\item The set of singular points $\Sigma$ is closed, and it has parabolic Hausdorff dimension at most $n-1$ \cite{FRS24}.
\end{itemize}

The proof of most of these results is based on blow-ups, i.e., considering limits 
\[\lim_{r\downarrow0} \frac{u(x_\circ+r x,t_\circ+r^2 t)}{r^2} \]
at any free boundary point $(x_\circ,t_\circ)$.
It turns out that the blow-up at \textit{any} singular point is a (unique) non-negative, 2-homogeneous, quadratic polynomial $p_{x_\circ,t_\circ}(x)$, and the singular set $\Sigma$ can be partitioned into the sets 
\[\Sigma_m:=\big\{(x_\circ,t_\circ)\in \Sigma : {\rm dim}(\{p_{x_\circ,t_\circ}=0\})=m\big\},\qquad m\in\{0,1,...,n-1\}.\]

In case of the set $\Sigma_{n-1}$, a very fine description of the solution and the free boundary near these points was recently established in \cite{FRS24}.
However, much less is known about $\Sigma_m$ for $m\leq n-2$, other than these sets have dimension at most $m$.

The only situation that has been quite well understood is the case of \textit{radial} solutions.
In such situation, the existence of a radial solution $u$, with an isolated singular point at $(0,0)$, was proved in \cite{HV97,AHV01,HR19}. For this solution, the ice region $\{u=0\}$ is a melting ball $\{|x|\leq\lambda(-t)\}$, for $t>0$, satisfying
\begin{eqnarray*}
\lambda(t) &\asymp &\sqrt{t}\, e^{-\sqrt{|\log t|/2}} \qquad \textrm{if}\quad n=2 \\
\lambda(t) &\asymp &\sqrt{t}\, |\log t|^{-\frac{1}{n-2}} \quad\ \textrm{if}\quad n\geq3.
\end{eqnarray*}

Our goal in this paper is to establish for the first time melting rates for \emph{all} solutions of \eqref{eq:obstacleProblem}.
Notice that, even at isolated free boundary points in $\Sigma_0$, solutions are not expected to be asymptotitcally radial, since the blow-up $p_{x_\circ,t_\circ}$ is in general not radial.
Thus, it is not clear a priori whether all solutions behave like those in the radial case or not.

\subsection{The 2D case}

Our best result is in dimension $n=2$, where we establish that at all singular points in $\Sigma_0$ the melting rates behave very much like the radial case.

Recall that in 2D there are only two types of singular points: those in $\Sigma_1$ (studied in \cite{FRS24}), and those in $\Sigma_0$ (covered by the following result).

\begin{thm}\label{main-thm1}
Let $u(x,t)$ be any solution of \eqref{eq:obstacleProblem} in $Q_1\subset\R^{2+1}$, and assume that $(0,0)$ is a singular free boundary point in $\Sigma_0$.

Then, for any $\delta>0$, in a neighborhood of the origin we have
\[ \left\{t< -C_1|x|^2 \exp\big(\big|\log |x|\big|^{\frac12+\delta}\big)\right\} \subset \{u=0\} \subset 
\left\{t< -c_1|x|^2\exp\big(\big|\log |x|\big|^{\frac12-\delta}\big)\right\} \]
for some positive constants $c_1,C_1$.

In particular, for any $\delta>0$ the free boundary near the origin satisfies
\[\partial\{u(\cdot,-t)>0\} \subset \left\{ c_1\sqrt{t}\,e^{-|\log t|^{\frac12+\delta}}<|x|<C_1\sqrt{t}\,e^{-|\log t|^{\frac12-\delta}} \right\}\]
for some positive constants $c_1,C_1$.
\end{thm}

The proofs of these upper and lower bounds for $\{u=0\}$ are completely different and independent from each other. 
Indeed, while the upper bound uses strongly the recent results in \cite{FRS24} and was more or less expected, the lower bound requires completely new ideas and does not use at all the techniques from \cite{FRS24}.
It is based on constructing explicit barriers and Harnack-type inequalities together with a delicate iterative method to show fine lower bounds for $u_t$ in parabolic cylinders $Q_{r_k}$ of size $r_k\asymp 2^{2^{-k}}$.

In both cases, an important starting point for these proofs is the regularity in time of solutions to \eqref{eq:obstacleProblem}, established by Caffarelli and Friedman in \cite{CF78,CF79}.

More precisely, they proved that (in dimension $n=2$) if $(0,0)$ is a singular point then $u_t\leq Ce^{-|\log r|^{\frac12-\delta}}$ in $Q_r$, for any $\delta>0$.
Here, $Q_r:=B_r\times(-r^2,0)$ denotes a parabolic cylinder.
This regularity is almost-optimal, in view of the radial example described above. 

If we \emph{assume} that a solution $u$ has the same regularity in time as the radial examples, then our proof yields that any such solution behaves exactly as the radial one.

\begin{prop}\label{prop-thm1}
Let $u$ be as in Theorem \ref{main-thm1}, and assume in addition that 
\begin{equation}\label{open-pb}
u_t \leq C e^{-C|\log r|^{\frac12}} \quad \textrm{in}\quad Q_r
\end{equation}
for all $r\in(0,1)$. 
Then, for any time $-t\in(-\frac12,0)$ we have
\[\partial\{u(\cdot,-t)>0\} \subset \left\{ c_1\sqrt{t}\,e^{-C_1|\log t|^{\frac12}}<|x|<C_1\sqrt{t}\,e^{-c_1|\log t|^{\frac12}} \right\}\]
for some positive constants $c_1,C_1$.
\end{prop}

It remains an open problem to decide whether \eqref{open-pb} holds for all solutions of \eqref{eq:obstacleProblem} in $\R^2$ or not.

\subsection{Higher dimensions}

The techniques we develop in this paper work not only in dimension 2 but also in arbitrary dimensions $n\geq3$.
In that case, the regularity in time of Caffarelli-Friedman \cite{CF78,CF79} yields that if $(0,0)$ is a singular point then $u_t\leq C|\log r|^{-\frac{2}{n-2}+\delta}$ in $Q_r$, for any $\delta>0$.
Using this, we prove the following:

\begin{thm}\label{main-thm2}
Let $u(x,t)$ be any solution of \eqref{eq:obstacleProblem} in dimension $n\geq3$, and assume that $(0,0)$ is a singular free boundary point in $\Sigma_0$.

Then, for any $\delta>0$ the free boundary at any time $-t\in(-\frac12,0)$ satisfies
\[\partial\{u(\cdot,-t)>0\} \subset \left\{ c_1\sqrt{t}\,e^{-|\log t|^{\delta}}<|x|<C_1\sqrt{t}\,|\log t|^{-\frac{1}{n-2}+\delta} \right\}\]
for some positive constants $c_1,C_1$.
\end{thm}

As before, if we assume in addition that
\begin{equation}\label{open-pb2}
u_t \leq C |\log r|^{-\frac{2}{n-2}} \quad \textrm{in}\quad Q_r
\end{equation}
for all $r\in(0,1)$, then for any time $-t\in(-\frac12,0)$ we have
\[\partial\{u(\cdot,-t)>0\} \subset \left\{ c_1\sqrt{t}\,|\log t|^{-C_1}<|x|<C_1\sqrt{t}\,|\log t|^{-\frac{1}{n-2}} \right\}\]
for some positive constants $c_1,C_1$.

It remains an open problem to decide whether \eqref{open-pb2} holds for all solutions of \eqref{eq:obstacleProblem} or not.

\subsection{Global regularity of the free boundary}

Concerning the intermediate strata $\Sigma_m$ with $1\leq m\leq n-2$, the possible behaviors of different solutions are expected to be very diverse, and thus it does not seem possible to establish matching upper and lower estimates like those in Theorem~\ref{main-thm1}.
In such case, we prove (non-optimal) extinction rates in Proposition \ref{prop:contact_set_ub} and Corollary \ref{cor:FB_almost_C2}.
These new rates allow us to establish for the first time the following global $C^1$ regularity of the free boundary:

\vspace{-0.06cm}
\begin{thm}\label{main-thm3}
Let $u(x,t)$ be any solution of \eqref{eq:obstacleProblem}, and assume that $(0,0)\in \Sigma_m$ is a singular free boundary point, $m\leq n-2$.

Then, if $U$ is a neighbourhood of $(0,0)$ such that $U\cap\Sigma_{n-1} = \emptyset$, the free boundary $\partial\{u>0\}\subset \R^n\times \R$ can be written as a $C^1$ graph $\{t=\Gamma(x)\}$ in $U$.

In particular, the whole free boundary is an $n$-dimensional $C^1$ manifold away from the set $\Sigma_{n-1}$.
\end{thm}

Notice that the free boundary $\partial\{u>0\}$ was proved to be locally Lipschitz (as a graph $\{t=\Gamma(x)\}$) by Caffarelli in \cite{Caf78}.
Moreover, it is easy to see that (in any dimension $n\geq1$) such a function $\Gamma$ is not $C^1$ at any point in $\Sigma_{n-1}$; see e.g. \cite{FRS24}.

Here, we prove for the first time that $\Gamma$ is actually $C^1$ at all singular points in $\Sigma_m$ with $m\leq n-2$, and therefore the free boundary is $C^1$ everywhere except at $\Sigma_{n-1}$.
Notice also that this is optimal since, in view of the extinction rates we prove in this paper, $\Gamma$ is not $C^2$ at any singular point.

\vspace{-0.14cm}
\subsection{Organization of the paper}
This paper is organized as follows.

We begin in Section \ref{sec:prelim} by introducing our setting and developing some technical tools such as estimates for the heat equation and the parabolic obstacle problem. Then, in Section \ref{sec:expansion} we derive an improved rate of convergence to the blow-up, Theorem \ref{thm:expansion}. 

Section \ref{sec:app} is devoted to proving an \textit{almost positivity property} for the heat equation, and in Section \ref{sec:homo} we construct self-similar solutions tailored to our domains. Finally, in Section \ref{sec:rates} we prove our main results.

\vspace{-0.14cm}
\subsection{Acknowledgements}

G.F. is member of the INDAM (“Istituto Nazionale di Alta Matematica”) research group GNAMPA and is supported by the GNAMPA-INDAM project \emph{PDE ellittiche che degenerano su varieta' di dimensione bassa e frontiere libere molto sottili}, E5324001950001.

X.R.-O. was supported by the EU under the ERC Consolidator Grant No\linebreak 101123223 (SSNSD), the AEI project PID2021-125021NAI00 funded by \linebreak MICIU/AEI/10.13039/501100011033 and by FEDER (Spain), the AEI-DFG project PCI2024-155066-2, the AGAUR Grant 2021 SGR 00087 (Catalunya), the AEI Grant RED2022-134784-T funded by MCIN/AEI/10.13039/501100011033 (Spain), and the AEI Maria de Maeztu Program for Centers and Units of Excellence in R\&D \linebreak CEX2020001084-M.

C.T.-L. has been funded by the European Research Council (ERC) under the Grant Agreement No 862342, from AEI project PID2021-125021NAI00 (Spain), and from the Grant CEX2023-001347-S funded by MICIU/AEI/10.13039/501100011033 (Spain).

\section{Preliminaries}\label{sec:prelim}

\subsection{Setting}
Throughout the paper, the \textit{spatial} dimension will be $n \geq 2$. Given $x \in \R^n$, we will sometimes denote $x = (\bar x,\bar y) \in \R^m\times\R^{n-m}$, with $m$ clear from the context. 

$B_r(x)$ will denote the ball of radius $r$ of $\R^n$, centered at $x$, and when $x$ is the origin we will simply write $B_r$. Moreover, $Q_r(x,t)$ will denote the parabolic cylinder $B_r(x)\times(t-r^2,t)$, and when $(x,t)$ is the origin we will just write $Q_r$.

We will use the notation $\p_p$ to denote the parabolic boundary of a set, that is, for a given set $E \in \R^{n+1}$, we denote by $\p_p E$ the set of points $(x_0,t_0)$ in $\p E$ such that for every $\varepsilon > 0$, $Q_\varepsilon(x_0,t_0)\not\subset E$.

We define the parabolic distance in $\R^{n+1}$ as
\begin{equation} \label{parabolic-d}
\operatorname{d_p}\left((x,t),(y,s)\right) :=  \sqrt{|x-y|^2+|t-s|}.
\end{equation}

\subsection{Estimates for the heat equation}
We will use the interior parabolic Harnack inequality:
\begin{thm}\label{thm:interior_Harnack}
    Let $u$ be a nonnegative solution to
    $$u_t - \Delta u = 0 \quad \text{in} \ Q_{1}.$$
    Then, for all $-1 < t_1 < t_2 \leq 0$,
    $$\sup\limits_{B_{1/2}}u(\cdot,t_1) \leq C\inf\limits_{B_{1/2}}u(\cdot,t_2),$$
    where $C$ depends only on $t_1$, $t_2$, and the dimension.
\end{thm}

We will also need the following quantitative version of the maximum principle.
\begin{lem}\label{lem:cylinder_measure}
    Let $u$ be a solution to
    $$\left\{\begin{array}{rclll}
        u_t - \Delta u & = & 0 & \text{in} & Q_1\\
        u & = & g & \text{on} & \p_pQ_1,
    \end{array}\right.$$
    and assume that $g \geq 0$ on $\p_pQ_1$, and that
    $$\left|\{g \geq 1\}\cap\p_p Q_1\right| \geq c_0$$
    for some $c_0 \in (0,1)$.
    Then,
    $$u(0,0) \geq \theta,$$
    where $\theta > 0$ depends only on $c_0$ and the dimension.
\end{lem}

\begin{proof}
By truncating $g$, we may assume without loss of generality that $0 \leq g \leq 1$. Now, note that there exists $c_n > 0$ such that
$\left|\p_pQ_1\cap\{t > -c_nc_0\}\right| \leq \frac{c_0}{3}$. Hence,
$$\left|\{g \geq 1\}\cap\p_pQ_1\cap\{t < -c_nc_0\}\right| \geq \frac{2c_0}{3}.$$

Then, we can write
\begin{align*}
    u(0,0) &= \int_{B_1\times\{-1\}}g\mathrm{d}\omega_i + \int_{\p B_1\times(-1,0)}g\mathrm{d}\omega_l\\
    &\geq \omega_i\big(\{g \geq 1\}\cap(B_1\times\{-1\})\big)+\omega_l\big(\{g \geq 1\}\cap(\p B_1\times(-1,-c_nc_0)\big),
\end{align*}
where $\omega_i,\omega_l$ is the caloric measure with base point the origin, supported on the initial and the lateral data.

Now, if $|\{g \geq 1\}\cap(\p B_1\times(-1,-c_nc_0)| \geq \frac{c_0}{3}$, we use \cite[Theorem 3.1]{FS83} to conclude that
\begin{align*}
    u(0,0) &\geq \omega_l\big(\{g \geq 1\}\cap(\p B_1\times(-1,-c_nc_0)\big)\\
    &\geq c'|\{g \geq 1\}\cap(\p B_1\times(-1,-c_nc_0)| \geq \frac{c'c_0}{3} > 0.
\end{align*}
On the other hand, we would have $|\{g(\cdot,-1) \geq 1\}\cap B_1| \geq \frac{c_0}{3}$. In this case, we use separation of variables to write
$$u(x,t) = \sum\limits_{n \geq 1}c_ne^{-\lambda_n(1+t)}\phi_n(|x|),$$
where $\{\phi_n\}$ is the orthornormal base of $L^2(B_1)$ formed by the eigenfunctions of the Laplacian. Then, since $\phi_1$ is radially decreasing, the minimum possible value of the following integral happens when $\{g(\cdot,-1) \geq 1\} = B_1 \setminus B_{r_*(c_0)}$, and thus
\begin{align*}
    c_1 &= \int_{B_1}g(x,-1)\phi_1(x)\mathrm{d}x \geq \int_{\{g(\cdot,-1)\geq 1\}\cap B_1}\tilde\phi_1(|x|)\\
    &\geq \int_{r_*(c_0)}^1\tilde\phi_1(r)|\p B_1|r^{n-1}\mathrm{d}r =: a(c_0) > 0,
\end{align*}
where $\tilde\phi_1(r) := \phi_1(re)$ for any $e \in \p B_1$. Therefore, $\|u(\cdot,-\frac{1}{2})\|_{L^2(B_1)} \geq a(c_0)e^{-\lambda_1/2}$, and since, by the maximum principle, $0 \leq u \leq 1$,
\begin{align*}
    \sup\limits_{B_r\times\{-\frac{1}{2}\}}u &\geq \frac{\|u(\cdot,-\frac{1}{2})\|_{L^2(B_r)}}{|B_r|^{1/2}} \geq \frac{\|u(\cdot,-\frac{1}{2})\|_{L^2(B_1)}-\|u(\cdot,-\frac{1}{2})\|_{L^2(B_1\setminus B_r)}}{|B_1|^{1/2}}\\
    &\geq \frac{a(c_0)e^{-\lambda_1/2}-(|B_1|-|B_r|)^{1/2}}{|B_1|^{1/2}} \geq \frac{a(c_0)e^{-\lambda_1/2}}{2|B_1|^{1/2}} =: b(c_0) > 0,
\end{align*}
choosing $r \in (0,1)$ appropriately close to $1$, only depending on $c_0$ and the dimension.

Finally, by the interior Harnack (Theorem \ref{thm:interior_Harnack}),
$$u(0,0) \geq c\sup\limits_{B_r\times\{-\frac{1}{2}\}}u \geq cb(c_0) > 0,$$
as we wanted to see.
\end{proof}

\subsection{The parabolic obstacle problem}
The regularity properties of solutions to \eqref{eq:obstacleProblem} have been established in the works of Caffarelli, and we summarize them in the following theorem.

\begin{thm}[\cite{Caf77}, \cite{CF79}]\label{T:caf77}
    Let $u$ be a solution to \eqref{eq:obstacleProblem} and suppose that $(0,0)\in\partial\{u>0\}$. Then,
    
    \noindent	\begin{itemize}
        \item[(i)] $u\in C^{1,1}_x(Q_{1/2})\cap C^{1}_t(Q_{1/2})$ and there exists $C>0$ depending only on $n$ such that
        $$\|D^2 u\|_{L^\infty(Q_{1/2})}+\|u_t\|_{L^\infty(Q_{1/2})}\le C\|u(\cdot,0)\|_{L^\infty(B_1)}$$	
        \item [(ii)] $\Sigma$ is relatively closed in $\partial\{u>0\}$.
        
        \item [(iii)] There exists $\varepsilon_0 > 0$ and $C$ depending only on $n$ and $\|u\|_{L^\infty}$ such that 
        \begin{equation}\label{eq:semiconvexity}
            D^2 u \ge -C|\log r|^{-\varepsilon_0}\quad\text{in }Q_r,
        \end{equation}
        for every $r\in (0,\frac12)$.
        
        \item[(iv)] For every 
        \[\gamma\in(0,{\textstyle\frac12}) \qquad \textrm{and} \qquad 0<\varepsilon<\frac{2}{n-2}\]
        there exists $C$ depending only on~$n$, $\varepsilon$, $\gamma$, and $\|u\|_{L^\infty}$  such that
        \begin{equation}\label{eq:continuity:u_t}
            u_t\le C\omega_n(r) \quad\text{in }Q_r,
        \end{equation}
        for all $r\in(0,\frac12)$, where 
        \begin{equation}\label{eq:modulus:u_t}
            \omega_n(r):=\begin{cases}
                2^{-|\log r|^{\gamma}} &\text{if } n=2\\	
                |\log r|^{-\varepsilon} & \text{if } n\ge3.
            \end{cases}
        \end{equation}
	
    \end{itemize}
\end{thm}

Next, we give the definition of regular and singular points on the free boundary $\partial\{{u>0}\}$. Let $u$ be a solution to \eqref{eq:obstacleProblem}, $(x_0,t_0)\in\partial\{u>0\}$ and define
    $$u_{x_0,t_0,r}(x,t):=r^{-2}u(x_0+rx,t_0+r^2t),$$
    which is also a solution  to \eqref{eq:obstacleProblem}. 
 \begin{itemize}
        \item[$\bullet$] $(x_0,t_0)$ is a \emph{regular point} if there exists $e \in \mathbb{S}^{n-1}$ such that
        $$u_{x_0,t_0,r}\to \frac{1}{2}(\max{0,e\cdot x})^2,\quad \text{as }r\to0^+,$$
        \item[$\bullet$]  $(x_0,t_0)$ is a \emph{singular point} if 
        $$u_{x_0,t_0,r}\to p_{2,x_0,t_0}\in \mathcal{P},\quad \text{as }r\to0^+,$$
        where
    $$\mathcal{P}:=\Big\{p(x)=\frac{1}{2}Ax\cdot x: A\in\R^{n,n}, A\ge0, \text{tr}(A)=1\Big\}.$$
    When $(x_0,t_0)=(0,0)$ we simply write $p_{2}=p_{2,x_0,t_0}$.
    \end{itemize}
     The convergence in both cases is locally uniformly in compact sets of $\mathbb{R}^n\times  \mathbb{R}$, and using the regularity estimates given by Theorem \ref{T:caf77} the convergence also holds in $C^1_\text{loc}$.
     
By the uniqueness of blow-ups at singular points (see \cite{Bla06}), one has that every free boundary point is either regular or singular. Furthermore, the set of regular points is relatively open in $\partial\{u>0\}$ and it is a $C^\infty$ manifold of dimension $n-1$. This last result was proved in \cite{Caf77,KN77}.

Finally, we define the singular set and the singular strata as follows:
    $$\begin{array}{lll}
    \Sigma & := & \big\{(x,t)\in \partial\{u>0\}: (x,t) \text{ is }\emph{singular}\big\},\\
    \Sigma_m & := & \big\{(x,t)\in\Sigma: \dim(\{p_{2,x,t}=0\})=m\big\}\\
    \Sigma^t & := & \big\{x: (x,t)\in\Sigma \big\},\\
    \Sigma_m^t & := & \big\{x: (x,t)\in\Sigma_m \big\},	
    \end{array}$$
for $m=0,\dots,n-1$.

\subsection{The second blow-up}
In this section, we state the characterization of the second blow-up, as established in \cite{FRS24}, where the authors prove some monotonicity formulae that enable them to derive this result for our equation.

First, we recall some definitions.
Let ${G}:\R^n\times(-\infty,0)\to\R$ be the Gaussian kernel for the heat operator, that is
$${G}(x,t):=\frac{1}{(-4\pi t)^{n/2}}\exp\left(  \frac{|x|^2}{4t} \right).$$
For every $w:\R^n\times(-1,1)\to\R$ and $r\in(0,1)$ let us define
\begin{align*}
    &D(r,w):=2r^2\int_{\{t=-r^2\}}|\nabla w|^2G,\\
    &H(r,w):=\int_{\{t=-r^2\}} w^2G,\\
    &\phi(r,w):=\frac{D(r,w)}{H(r,w)}.
\end{align*}

We have the following bounds relating $H$ and the norms of $u - p_2$.
\begin{lem}[\cite{FRS24}, Corollary 6.2, Lemma 6.3]
    Let $u$ be a solution to \eqref{eq:obstacleProblem}, $(0,0)\in\Sigma$. Then, there exists $C>1$ depending only on $n$ and $\|u\|_{L^\infty}$ such that for all $r \in (0,\frac{1}{2})$,
    
    \begin{equation}\label{eq:L2:Linfty}
        \|u-p_2\|_{L^\infty(Q_r)}\le C \|u-p_2\|_{L^2(Q_{2r})},
    \end{equation}
    
    \begin{equation}\label{eq:equivalent:norm}
        C^{-1}H(r,\xi (u-p_2))^{1/2}\le \|u-p_2\|_{L^2(Q_r)}\le C H(r,\xi (u-p_2))^{1/2}.
    \end{equation}
\end{lem}

Let $\xi\in C_c^\infty(B_{1/2})$ let be a spatial cut-off function such that $\xi\equiv1$ in $B_{1/4}$ and $\xi\ge0$. Then, we have a characterization of the second blow-ups.

\begin{prop}[\cite{FRS24}, Lemma 5.8, Corollary 5.9, Proposition 6.7]\label{pro:classification}
    Let $u$ be a solution to \eqref{eq:obstacleProblem}, $(0,0)\in\Sigma$, $w:=u-p_2$, $m=\dim(\{p_2=0\})$. Then, the following limit exists
    $$\lim_{r\to0^+}\phi(r,w\xi):=\lambda_*.$$
    Set
    \[
    \tilde{w}_r:=\frac{w(r\cdot,r^2\cdot)}{H(r,\xi w)^{1/2}}.
    \]
    Then, for every $r_k\to0^+$, there exists a subsequence $r_{k_l}$ such that 
    $$\tilde w_{r_{k_l}}\to q,\quad \nabla\tilde w_{r_{k_l}} \rightharpoonup \nabla q  \quad \text{in } L^2_{loc}(\R^n\times(-\infty,0]),$$
    where $q\not\equiv0$ is $\lambda_*$-homogeneous. Moreover, we have that
    \begin{itemize}
        \item[(a)] If $m\in\{0,\dots,n-2\}$ then $\lambda_*=2$ and, in some coordinates, 
        \begin{equation*}
            p_2(x)=\frac{1}{2}\sum_{i=m+1}^n \mu_ix_i^2,\quad q(x,t)=At+\nu \sum_{i=m+1}^n x_i^2-\sum_{i=1}^{m} \nu_i x_i^2,
        \end{equation*}
        where $\mu_i>0$, $A\ge0$, $\nu\ge0$, $\sum_{i=m+1}^n \mu_i =1$, and $$A-2(n-m)\nu+2\sum_{i=1}^{m} \nu_i=0,$$
        which means that $q_t-\Delta q=0$. Moreover, there exist constants $0<c_1\le c_2$ such that        
        $$c_1\le\|q\|_{L^2(Q_1)}\le c_2.$$
        \item[(b)] If $m=n-1$ then $\lambda_*\in[2+\alpha_0,3]$ for some constant $\alpha_0=\alpha_0(n)\in(0,1]$ and $q$ solves the parabolic thin obstacle problem.
    \end{itemize}
    
\end{prop}

\section{Expansion of solutions at singular points}\label{sec:expansion}

The goal of this section is to show an expansion of solutions at singular points, by using the semiconvexity of the solutions and a compactness argument (see \cite{FS19} for the elliptic case). 

\begin{thm}\label{thm:expansion}
    Let $u$ be a solution to \eqref{eq:obstacleProblem} such that $(0,0)\in\Sigma$. Let $\gamma \in (0,\frac12)$ and $0<\varepsilon<\frac{2}{n-2}$.
    Then, there exists a constant $C>0$ depending only on $n$, $\gamma$, $\varepsilon$, and $\|u\|_{L^\infty}$, such that 
    \begin{equation*}
        |u(x,t)-p_2(x)|\le C(|x|^2+|t|)\sigma\big(\sqrt{|x|^2+|t|}\big)\quad \text{in} \quad Q_{1/2},
    \end{equation*}
    where $\sigma:\R_+\to \R_+$ is defined by
    \begin{equation*}
        \sigma(r)=\begin{cases}
            r^{\alpha_0}     &\text{if }      (0,0)\in\Sigma_{n-1},       \\
            |\log r|^{-\varepsilon_0} &\text{if }   (0,0)\in\Sigma_{m}, \text{ for } m\in\{1,\dots,n-2\},        \\
            2^{-|\log r|^\gamma}      &\text{if }      (0,0)\in\Sigma_{0} \text{ and }n=2\\  
            |\log r|^{-\varepsilon}      &\text{if }      (0,0)\in\Sigma_{0} \text{ and }n\ge 3,
        \end{cases}
    \end{equation*}	
    where $\alpha_0 \in (0,1]$ is the constant as in Proposition \ref{pro:classification} $(b)$, $\varepsilon_0>0$ is a small constant given by \eqref{eq:semiconvexity}.
\end{thm}

\begin{proof}
    By using \eqref{eq:L2:Linfty}, it is enough to prove that for every $r\in(0,1/2)$ it holds
    \begin{equation}\label{eq:l2:expansion}
        \|u(r\cdot,r^2\cdot)-p_{2}(r\cdot)\|_{L^2(Q_2)}\le Cr^2\sigma(r).
    \end{equation}
   
    \noindent
    \emph{Case 1.} Let $(0,0)\in\Sigma_{n-1}$. By using \eqref{eq:equivalent:norm} and \cite[Lemma 5.6(b)]{FRS24}, we obtain that
\begin{align*}
\| u(r \cdot,r^2 \cdot)-p_{2}(r\cdot)& \|_{L^2(Q_2)}\le C H(2r,\xi(u-p_2))^{1/2}\\
&\le Cr^{\lambda_*}( H(1,\xi(u-p_2))^{1/2}+1)^{1/2}\le Cr^{\lambda_*},
\end{align*}   
where $\lambda_*:=\lim_{r\to0^+}\phi(r,w\xi)$ and the constant $C>0$ depends only on $n$ and $\|u\|_{L^\infty}$. By using Proposition \ref{pro:classification}, it follows that the frequency $\lambda_*\geq2+\alpha_0$ and \eqref{eq:l2:expansion} follows in the case $\Sigma_{n-1}$.

    \bigskip

    \noindent
    \emph{Case 2.} Let us suppose now that $(0,0)\in\Sigma_{m}$ for $m\in\{1,\dots, n-2\}$. We define $L_0:=\{x:p_2(x)=0\}$, and have that
    \begin{equation}\label{eq:ar}
        a_r:= \| r^{-2}u(r \cdot,r^2 \cdot)-p_{2}(\cdot) \|_{L^2(Q_2)}=o(1),
    \end{equation}
    by the blow-up convergence. 
    
    
    By contradiction, let us suppose that there exists $r_k\downarrow0$ and a constant $M>1$ to be chosen later such that
    \begin{equation}\label{eq:contradiction1}
        a_{r_k}\ge M |\log r_k|^{-\varepsilon_0}.
    \end{equation}
    By using \eqref{eq:semiconvexity} and \eqref{eq:continuity:u_t}, we have that
    \begin{equation}\label{eq:*1}
        \begin{cases}
        \partial_{t}\big(r^{-2}u(r \cdot,r^2 \cdot)-p_{2}(\cdot)\big)=\partial_{t}u(r \cdot,r^2 \cdot)\le C \omega_n(r) \le C|\log r|^{-\varepsilon_0},\\
            \partial_{ee}\big(r^{-2}u(r \cdot,r^2 \cdot)-p_{2}(\cdot)\big)=\partial_{ee}u(r \cdot,r^2 \cdot)\ge-C|\log r|^{-\varepsilon_0},
            \end{cases}
    \end{equation}
    for every $e\in L_0\cap\mathbb{S}^{n-1}$.

    Next, by using \eqref{eq:equivalent:norm} and \eqref{eq:contradiction1} we get that there exists $C_2>0$ depending only on $n$ and $\|u\|_{L^\infty}$, such that
    \begin{equation}\label{eq:*2}
        H(r_k,\xi(u-p_2))^{1/2}\ge C_2 a_{r_k}r_k^2\ge C_2 M r_k^2|\log r_k|^{-\varepsilon_0}.
    \end{equation}
    Let us define
    \begin{equation}\label{eq:tilde:w}
        \tilde{w}_{r_k}:=\frac{u(r_k\cdot,r_k^2\cdot)-p_2(r_k\cdot)}{H(r_k,\xi(u-p_2))^{1/2}}.
    \end{equation}
    By applying Proposition \ref{pro:classification}, it follows that, up to a subsequence,
    $$\tilde{w}_{r_k}\to q \text{ in }L^2(Q_1),$$
    where $q$ is a parabolic $2$-homogeneous polynomial satisfying
    \begin{equation}\label{eq:q}
        D^2 q_{|_{L_0^\perp}}\ge0,\quad q_t\ge0, \quad\partial_tq -\Delta q=0,\quad 0<\bar{c}_1\le\|q\|_{L^2(Q_1)}\le\bar{c}_2.
    \end{equation}
    Moreover, by using \eqref{eq:*1} and \eqref{eq:*2}, we get that
    \[
        \partial_{ee}\tilde{w}_{r_k}=\frac{r_k^2 \partial_{ee}u(r_k\cdot,r_k^2\cdot)}{H(r_k,\xi(u-p_2))^{1/2}}\ge  -\frac{Cr_k^2|\log r_k|^{-\varepsilon_0}}{H(r_k,\xi(u-p_2))^{1/2}}\ge \frac{-C}{C_2M},
    \]
    for every $e\in L_0\cap\mathbb{S}^{n-1}$, and analogously 
    \[
        \partial_{t}\tilde{w}_{r_k}\le \frac{C}{C_2M}.
    \]
    Taking the limit in the previous two inequalities we get 
    \begin{equation}\label{eq:*5}
        \begin{cases}
            \partial_{ee}q \ge -\frac{C}{C_2M}, &\forall e\in L_0\cap\mathbb{S}^{n-1},\text{ in } Q_1,\\
            \partial_{t}q \le \frac{C}{C_2 M}, &\text{ in } Q_1.
        \end{cases}
    \end{equation}
    Now, we claim that there exists a constant $C_1>0$ such that
    one of the following holds:
    \begin{align}\label{eq:xxx}
        \begin{aligned}
            &(a) \quad\text{there exists } \tilde{e}\in L_0\cap\mathbb{S}^{n-1} \text{ such that } \min_{Q_1}\partial_{\tilde{e}\tilde{e}}q \le -C_1,\\
            &(b) \quad \max_{Q_1}q_t\ge C_1.
        \end{aligned}
    \end{align}  
    By contradiction, let us suppose that there exists a sequence $q^{(j)}$ satisfying \eqref{eq:q} and 
    \[
        q_t^{(j)}\le\frac{1}{j}, \quad \partial_{\tilde{e}\tilde{e}}q^{(j)}\ge-\frac{1}{j},\quad  \forall\tilde{e}\in L_0\cap\mathbb{S}^{n-1}.
    \]
    Since $q^{(j)}$ belongs to a finite dimensional space and $\|q^{(j)}\|_{L^2(Q_1)}\le \bar{c_2}$, by compactness, up to a subsequence, there exists a limit parabolic $2$-homogeneous polynomial $q^{(\infty)}$ satisfying
    \begin{align*}
        &q_t^{(\infty)}\le0, \quad \partial_{\tilde{e}\tilde{e}}q^{(\infty)}\ge0,\quad  \forall\tilde{e}\in L_0\cap\mathbb{S}^{n-1},\\
        &D^2 q^{(\infty)}_{|_{L_0^\perp}}\ge0, \quad q_t^{(\infty)} -\Delta q^{(\infty)}=0,\quad \|q^{(\infty)}\|_{L^2(Q_1)}\ge\bar{c}_1,
    \end{align*}
    which is a contradiction, so the claim is true. Then, by choosing $M$ big enough in \eqref{eq:*5} we get a contradiction with \eqref{eq:xxx} and \eqref{eq:l2:expansion} holds true for $m=1,\dots,n-2$.
    
    \bigskip

    \noindent
    \emph{Case 3.}  Finally, let us suppose that $(0,0)\in\Sigma_{0}$ and define $a_r$ as in \eqref{eq:ar}. By contradiction, let us suppose that 
    \[
        a_{r_k}\ge M \omega_n(r_k)
    \]
    for a sequence $r_k\downarrow0$ and a constant $M>1$ to be chosen later, where $\omega_n$ is given by~\eqref{eq:modulus:u_t}.
    
    By \eqref{eq:continuity:u_t} it follows that
    $$\partial_{t}\big(r^{-2}u(r \cdot,r^2 \cdot)-p_{2}(\cdot)\big) =u_t (r\cdot,r^2\cdot) \le C\omega_n( r),$$
    Defining $\tilde{w}_{r_k}$ as in \eqref{eq:tilde:w} and by using the same computations as in \emph{Case 2}, we get that
    \[
        \partial_{t}\tilde{w}_{r_k}\le \frac{C}{C_2M}.
    \]
    This inequality, combined with Proposition \ref{pro:classification}, allows us to conclude that
    $\tilde{w}_{r_k}\to q$ in $L^2(Q_1)$ where $q$ is a parabolic $2$-homogeneous polynomial satisfying
    \[
        0\le q_t \le \frac{C}{C_2 M},\quad D^2 q\ge0, \quad\partial_tq -\Delta q=0,\quad 0<\bar{c}_1\le\|q\|_{L^2(Q_1)}\le\bar{c}_2.
    \]
    From this point on, arguing as in \emph{Case 2} with minor differences, we  get a contradiction and our statement follows in the case $\Sigma_0$.   
\end{proof}

As an immediate consequence of the expansion theorem, we obtain the following regularity characterization of the \emph{spatial} part of the singular set, which improves the result proved in \cite{LM15}, by showing an explicit modulus of continuity.

\begin{cor}
    Let $u$ be any solution to \eqref{eq:obstacleProblem}, $\alpha_0>0$ be the constant from Proposition~\ref{pro:classification}~$(b)$ and $\varepsilon_0>0$ as in \eqref{eq:semiconvexity}. Then the following holds true.
    \begin{enumerate}
        \item[(a)] $\Sigma_{n-1}^t$ is locally contained in a $C_x^{1,\alpha_0}$ manifold of dimension $(n-1)$.
        \item[(b)] If $m\in\{1,\dots,n-2\}$, $\Sigma_{m}^t$ is locally contained in a $C_x^{1,\log^{\varepsilon_0}}$ manifold of dimension $m$.
        \item[(c)] $\pi_x(\Sigma_{n-1}\cap Q_{1/2})$ is locally contained in a $C_x^{1,\alpha_0}$ manifold of dimension $(n-1)$.
        \item[(d)] If $m\in\{1,\dots,n-2\}$, $\pi_x(\Sigma_{m}\cap Q_{1/2})$ is locally contained in a $C_x^{1,\log^{\varepsilon_0}}$ manifold of dimension $m$.
    \end{enumerate}
\end{cor}

\begin{proof}
    We prove statements $(a)$ and $(b)$ together. Let us fix $t_0\in(-1/4,0]$ and define $S_{\lambda_m}^{t_0}:=\{x\in\Sigma^{t_0}:\lambda_*\ge\lambda_m\}$, where $\lambda_m=2$ if $ m\in\{1,\dots,n-2\}$ and $\lambda_m=2+\alpha_0$ if $ m=n-1$.
    
    By \cite[Lemma 7.4]{FRS24}, we have that the map 
    \begin{equation}\label{eq:upper:sc}
        \Sigma\ni (x_0,t_0)\to \phi\big(0^+,\xi\big(u(x_0+r\cdot,t_0+r^2\cdot)-p_{2,x_0,t_0}\big)\big)
    \end{equation}
    is upper semicontinuous. Consequently, the map
    $$\Sigma^{t_0}\ni x_0\to \phi\big(0^+,\xi\big(u(x_0+r\cdot,t_0+r^2\cdot)-p_{2,x_0,t_0}\big)\big)$$
    is also upper semicontinuous. This implies that $S^{t_0}_{\lambda_m}$ is closed, so $K:= S_{\lambda_m}^{t_0}\cap \overline{B_{1/4}} $ is compact. 
    Moreover, by Proposition \ref{pro:classification} we get $ \overline{\Sigma_m^{t_0}\cap B_{1/4}}\subset K$.
    
    Let us define
    \[
    \gamma_{m}(r):=\begin{cases}
        r^{\alpha_0} &  \text{if } m=n-1\\
        |\log r|^{-\varepsilon_0} &  \text{if } m\in\{1,\dots,n-2\}
    \end{cases}
    \]

    For $x_0\in K$, set 
    $$P_{x_0}(x,t):=p_{2,x_0,t_0}(x-x_0).$$
    We claim that $K,f\equiv0$ and $\{P_{x_0}\}_{x_0\in K}$ satisfy the assumptions of the  Whitney's extension Theorem (see \cite[Lemma 3.10]{FS19}), that is,\\
    $(i)$ $P_{x_0}(x_0)=0$,\\
    $(ii)$ there exists a constant $C>0$, depending only on $n$ and $\|u\|_{L^\infty}$, such that
    \begin{equation*}
        |   D^k P_{x_0} (x)  -   D^k P_{x} (x) |\le C |x-x_0|^{2-k}\gamma_m(|x-x_0|),
    \end{equation*}
    for all $x,x_0\in K$ and $k\in\{0,1,2\}$.
    
    The condition $(i)$ is trivially verified. 
    
    Next, given $x,x_0\in K$, set $|x-x_0|:=r\le 1/2$ and for simplicity of notation take $x_0=0$ and $t_0=0$. Noticing that $Q_1\subset Q_2(x/r)$, it follows that
    \begin{align*}
        &\hspace{-10mm}\|(P_{0}-P_{x})(r\cdot)\|_{L^2(Q_1)} \leq \\
        &\le 
        \|u(r\cdot,r^2\cdot)-P_{0}(r\cdot)\|_{L^2(Q_1)}+
        \|u(r\cdot,r^2\cdot)-P_{x}(r\cdot)\|_{L^2(Q_1)}\\
        &= \|u(r\cdot,r^2\cdot)-p_{2}(r\cdot)\|_{L^2(Q_1)}
        + \|u(r\cdot,r^2\cdot)-p_{2,x,0}(r\cdot-x)\|_{L^2(Q_1)}\\
        &\le \|u(r\cdot,r^2\cdot)-p_{2}(r\cdot)\|_{L^2(Q_1)}
        +\|u(r\cdot,r^2\cdot)-p_{2,x,0}(r\cdot-x)\|_{L^2(Q_2(x/r))}\\
        &\le \|u(r\cdot,r^2\cdot)-p_{2}(r\cdot)\|_{L^2(Q_1)}
        +  \|u(x+r\cdot,r^2\cdot)-p_{2,x,0}(r\cdot)\|_{L^2(Q_2)}\\
        &\le C r^2\gamma_m(r),
    \end{align*}
    where in the last inequality we have used \eqref{eq:l2:expansion}. Since the space of parabolic \mbox{2-homogeneous} polynomials is finite dimensional, the  norms $\|\cdot\|_{L^2(Q_1)}$ and $\|\cdot\|_{C^k(B_1)}$ are equivalent. Then, the property $(ii)$ is also satisfied and by applying the Whitney's extension Theorem (see \cite[Lemma 3.10]{FS19}) we get that there exists a function $F\in C^{2,\gamma_m}(\R^n)$ such that
    $$F(x)=P_{x_0}(x)+|x-x_0|^2\gamma_m(|x-x_0|),\quad\forall x,x_0\in K.$$
    In addition, we have that
    $$
    \Sigma_m^{t_0}\cap B_{1/4}\subset K \subset \{\nabla F=0\}.
    $$
    Hence, given $x_0\in \Sigma_m^{t_0}\cap B_{1/4}$, it follows that $\nabla F(x_0)=0$ and 
    $$\dim\ker(D^2 F(x_0))= \dim (\{p_{2,x_0,t_0}=0\})=m.$$
    Then, up to a change of coordinates, we have that $|D^2_{(x_1,\dots,x_{n-m})} F (x_0)|\not=0$. Hence, by applying the Implicit Function Theorem, $\bigcap_{i=1}^{n-m}\{\partial_{x_i}F=0\}$ is a $m$-dimensional manifold of class $C^{1,\gamma_m}$ which contains $\Sigma_m^{t_0}\cap B_{1/4}$, and this concludes the proof of statements $(a)$ and $(b)$.
    
    \medskip
    
    Next, we prove $(c)$ and $(d)$. 
    Let us define $S_{\lambda_m}=\{(x,t)\in\Sigma:\lambda_*\ge\lambda_m\}$. By the upper semicontinuity of \eqref{eq:upper:sc}, we have that $S_{\lambda_m}$ is closed, so $K:=\pi_x(\Sigma_{\lambda_m}\cap\overline{Q_{1/4}})$ is a compact set and by using Proposition \ref{pro:classification}, $ \pi_x(\Sigma_m \cap\overline{Q_{1/4}})\subset K$.

    For $x_0\in K$, there exists $t_0\in(-1/16,0]$ such that $(x_0,t_0)\in\Sigma$. Defining $$P_{x_0}:=p_{2,x_0,t_0}(x-x_0),$$  
    we claim that $K,f\equiv0$ and $\{P_{x_0}\}_{x_0\in K}$ satisfy the assumptions $(i)$ and $(ii)$ of the  Whitney's extension Theorem.
    
    The first condition is verified by definition.
    
    Given $x,x_0\in K$, set $r:=|x-x_0|\le 1/2$. By definition there exist $t,t_0$ such that $(x,t),(x_0,t_0)\in\Sigma$. Without loss of generality, assume that $(x_0,t_0)=(0,0)$ and $t\le0$. Set $c_r:= O(r^2\gamma_m(r))$. By using Theorem \ref{thm:expansion}, we get
    \begin{align}
        \begin{aligned}\label{eq:exp:1}
            &u(r\cdot,0)=p_2(r\cdot)+c_r,\quad\text{ in }B_3,\\ 
            &u(x+r\cdot,t)=p_{2,x,t}(r\cdot)+c_r,\quad\text{ in }B_3. 
        \end{aligned}
    \end{align}
    Noticing that $r=|x|$ and $|x+ry|\le 4r$, for every $y\in B_3$, \eqref{eq:exp:1} implies that
    \begin{align*}
        \begin{aligned}
            &u(x+ry,0)=p_2(x+ry)+O\left(|x+ry|^2 \gamma_m (|x+ry|)\right)  \\
            &= p_2(x+r y)+ c_r, \quad \text{ for }y\in B_3.
        \end{aligned}
    \end{align*}
    In addition, since $u_t\ge 0$ and $t\le0$, we have that
    \begin{align*}
        \begin{aligned}
            0\le u(x+ry,0)-u(x+ry,t)=p_2(x+ry)-p_{2,x,t}(ry)+c_r
        \end{aligned}
    \end{align*}
    for $y\in B_3$.
    
    Now, we observe that 
    \begin{align*}
        p_2(x)-p_{2,x,t}(0)\ge 0,\\
        p_2(0)-p_{2,x,t}(-x)\le 0,
    \end{align*}
    so there exists a point $\bar{x}$ belongs to the segment connecting $0$ and $-x/r\in B_1$ such that $p_2(x+r\bar{x})-p_{2,x,t}(r\bar{x})= 0$.

    Since
    $$\Delta (p_2(x+r\cdot )-p_{2,x,t}(r\cdot))=0,$$
    by applying the Harnack inequality to $p_2(x+r\cdot)-p_{2,x,t}(r\cdot)+c_r$ (which is nonnegative), we obtain
    \begin{align*}
        \| p_2(x+r\cdot)-p_{2,x,t}(r\cdot)+c_r \|_{L^\infty(B_2)}   \le C \inf_{B_{3}}\left( p_2(x+r\cdot)-p_{2,x,t}(r\cdot)+c_r \right)\le Cc_r,
    \end{align*}
    which implies
    \begin{align*}
        &\| P_0(r\cdot)-P_{x}(r\cdot)\|_{L^\infty(B_{1})}\le  \| P_0(r\cdot)-P_{x}(r\cdot)\|_{L^\infty(B_{2}(-x/r))}\\
        &=\| P_0(x+r\cdot)-P_{x}(x+r\cdot)\|_{L^\infty(B_2)} \le (C+1)c_r.
    \end{align*}
    Hence, the assumptions of the Whitney's extension theorem are satisfied, allowing us to conclude as for statements 
    $(a)$ and $(b)$.    
\end{proof}

\section{A parabolic almost positivity property}\label{sec:app}

The goal of this section is to prove an \textit{almost positivity property} for solutions to the heat equation in our domains of interest. 

\begin{defn}\label{defn:D_eta}
    Let $\eta > 0$ and $m \in \{0,\ldots,n-2\}$, and recall that for $x \in \R^n$, we write $\bar x = (x_1,\ldots,x_m)$, $\bar y = (x_{m+1},\ldots,x_n)$. Then, we denote
    $$D_{\eta,m} := \left\{|\bar y| > \eta|t|^{1/2}\right\}\cap\left\{|\bar y| > \eta|\bar x|\right\}.$$
   Notice that when $m=0$ we simply have $\bar y=x$ and hence $D_{\eta,0} := \left\{|x| > \eta|t|^{1/2}\right\}$.
\end{defn}

\begin{figure}[H]
	\centering
	\includegraphics{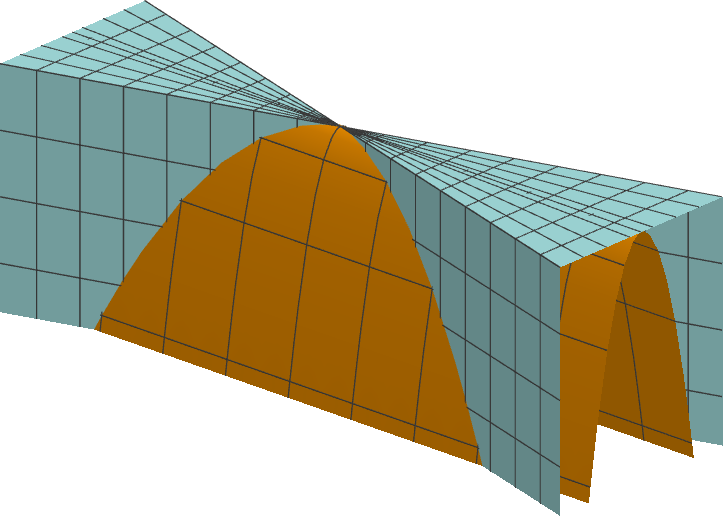}
	\caption{The complement of $D_{\eta,1}$ with $n = 2$. The orange surface is $\{|y| = \eta|t|^{1/2}\}$, and the blue one is $\{|y| = \eta|x|\}\cup\{|y|\leq\eta|x|,t=0\}$.}
\end{figure}

The main result of this Section is the following.

\begin{prop}\label{cor:almost_positivity}
    Let $\eta \in (0,\frac{1}{4})$, and $D_{\eta,m}$ as in Definition \ref{defn:D_eta}. There is a constant $\nu>0$ such that the following holds.
    
    Let $u$ satisfy
    $$\left\{\begin{array}{rclll}
        u_t - \Delta u & = & 0 & \text{in} & D_{\eta,m}\cap Q_1\\
        u & \geq & \!-\nu & \text{in} & D_{\eta,m}\cap Q_1\\[0.1cm]
        u & \geq & 1 & \text{on} & \left\{|\bar y| = \frac{1}{2}\right\}\cap Q_1\\[0.1cm]
        u & \geq & 0 & \text{on} & \p D_{\eta,m}\cap Q_1.
    \end{array}\right.$$
    Then,
    $$u \geq 0 \quad \text{in} \quad D_{\eta,m}\cap Q_{1/2}.$$
    The constant $\nu$  depends only on $n$ (and not on $\eta$).
\end{prop}

Our approach to establish this result is in the spirit of De~Silva and Savin in their proofs of boundary Harnack inequalities \cite{DS20,DS22}.

Actually, in order to provide a result that can be later reused in different settings, we distill the conditions used in the proof in the following characterization.

\begin{defn}\label{defn:boundary_accessible}
    Let $\Omega\subset \R^n\times\R$ be an open set where the Dirichlet problem for the heat equation is well posed. Denote $\Omega_r = \{x \in \Omega : \operatorname{d_p}(x,\p\Omega) \geq r\}$, where $\operatorname{d_p}$ is the parabolic distance \eqref{parabolic-d}. 
    We say $\Omega$ is \textit{parabolically accessible} if there exist $\delta_0, c_0 > 0$, and $N \in \N$, such that, for every $\delta \in (0,\delta_0)$,
    \begin{itemize}
        \item[(i)] For every $(x_0,t_0) \in \Omega_{\delta/2}\cap Q_{1-\delta}$, there exist $(x_i,t_i) \in \Omega$, with $i = 1,\ldots,N$, such that $(x_N,t_N) \in \Omega_\delta$, and for all $i = 0,\ldots,N$, $t_{i+1} \leq t_i$,
        $$Q_{2r_i}(x_i,t_i) \subset \Omega, \quad \text{and} \quad |x_i-x_{i+1}| \leq r_i,$$
        where $r_i = \sqrt{t_{i+1} - t_i}$.
        \item[(ii)] For every $(x_0,t_0) \in \Omega\cap Q_{1-2\delta}$,
        $$\big|\p_pQ_{2\delta}(x_0,t_0) \cap (\Omega_\delta\cup\Omega^c)\big| \geq c_0|\p_pQ_{2\delta}|.$$
    \end{itemize}
\end{defn}

We start with a quantitative \textit{propagation of positivity} property.

\begin{lem}\label{lem:almost_positivity_iteration}
    Let $\Omega$ be parabolically accessible in the sense of Definition \ref{defn:boundary_accessible}.
    Then, there exist $\mu,\delta>0$ such that the following holds.
    
    Let $u$ satisfy
    $$\left\{\begin{array}{rclll}
        u_t - \Delta u & = & 0 & \text{in} & \Omega\cap Q_1\\
        u & \geq & \!-\mu & \text{in} & \Omega\cap Q_1\\
        u & \geq & 1 & \text{in} & \Omega_\delta\cap Q_1\\
        u & \geq & 0 & \text{on} & \p\Omega\cap Q_1.
    \end{array}\right.$$
Then,
    $$\left\{\begin{array}{rclll}
        u & \geq & \!-\mu^2 & \text{in} & \Omega\cap Q_{1/2}\\
        u & \geq & \mu & \text{in} & \Omega_{\delta/2}\cap Q_{1/2}.
    \end{array}\right.$$
    The constants $\mu$ and $\delta$ depend only on the constants in Definition \ref{defn:boundary_accessible} and $n$.
\end{lem}

\begin{proof}		
    First, let $\mu \in (0,1)$ to be chosen later. 
    We will prove that for small enough $\delta \in (0,\delta_0)$, we have $u \geq -\mu^2$ in $\Omega\cap Q_{1/2}$.
    
    For this, let $(x_0,t_0) \in \Omega\cap Q_{1-2\delta}$. Let $v$ be the solution to
    $$\left\{\begin{array}{rclll}
        v_t - \Delta v & = & 0 & \text{in} & Q_{2\delta}(x_0,t_0)\\
        v & = & \!-u^- & \text{on} & \p_p Q_{2\delta}(x_0,t_0).
    \end{array}\right.$$
    Then, $u \geq v$ in $Q_{2\delta}(x_0,t_0)$, and in particular at $(x_0,t_0)$.
    
    By construction, $v \geq -\mu$ on $\p_p Q_{2\delta}(x_0,t_0)$. Moreover, condition (ii) in Definition~\ref{defn:boundary_accessible} implies that
    $$\left|\p_pQ_{2\delta}(x_0,t_0) \cap (\Omega_\delta\cup\Omega^c)\right| \geq c_0|\p_pQ_{2\delta}|,$$
    and hence
    $$\left|\{v \geq 0\}\cap\p_pQ_{2\delta}(x_0,t_0)\right| \geq c_0|\p_pQ_{2\delta}|,$$
    and by Lemma \ref{lem:cylinder_measure}, $v(x_0,t_0) \geq -(1-\theta)\mu$, with $\theta \in (0,1)$.
    
    Repeating this argument, we obtain that $u \geq -(1-\theta)^k\mu$ in $Q_{1-2k\delta}$, and then choosing $k$ such that $(1-\theta)^k \leq \mu$ and $\delta$ such that $k\delta \leq \frac{1}{4}$, the conclusion follows.
    
    Now we use a parabolic Harnack chain to prove the second inequality. Let $(x_0,t_0) \in \Omega_{\delta/2}\cap Q_{1/2}$. Then, there exist $(x_i,t_i) \in \Omega$, with $i = 1,\ldots,N$, such that
    $$Q_{2\sqrt{t_i-t_{i+1}}}(x_i,t_i) \subset \Omega \quad \text{and} \quad |x_i-x_{i+1}| \leq \sqrt{t_i-t_{i+1}}.$$
    Then, by the parabolic Harnack inequality (Theorem \ref{thm:interior_Harnack}), applied to $u+\mu$,
    $$u(x_i,t_i)+\mu \geq c\big(u(x_{i+1},t_{i+1})+\mu\big),$$
    with a uniform dimensional constant, and it follows that 
    $$u(x_0,t_0)+\mu \geq c^N\big(u(x_N,t_N)+\mu\big) \geq c^N(1+\mu),$$
    because $(x_N,t_N) \in \Omega_\delta$. Choosing $\mu = c^N/2$ completes the proof.
\end{proof}

Then, we iterate the previous lemma.  

\begin{prop}\label{prop:almost_positivity_bulk}
    Let $\Omega$ be parabolically accessible in the sense of Definition \ref{defn:boundary_accessible}.
    Then, there exist $\mu,\delta>0$ such that the following holds.
    
    Let $u$ satisfy
    $$\left\{\begin{array}{rclll}
        u_t - \Delta u & = & 0 & \text{in} & \Omega\cap Q_1\\
        u & \geq & \!-\mu & \text{in} & \Omega\cap Q_1\\
        u & \geq & 1 & \text{in} & \Omega_\delta\cap Q_1\\
        u & \geq & 0 & \text{on} & \p\Omega\cap Q_1.
    \end{array}\right.$$
    Then,
    $$u \geq 0 \quad \text{in} \quad \Omega\cap Q_{1/2}.$$ 
    
    The constants $\mu$ and $\delta$ depend only on the constants in Definition \ref{defn:boundary_accessible} and $n$.
\end{prop}

\begin{proof}
    First, iterating Lemma \ref{lem:almost_positivity_iteration} gives $u \geq \mu^k$ in $\Omega_{2^{-k}\delta}\cap Q_{2^{-k}}$ for all $k\geq1$.
    In particular, since 
    $$\cup_{k\geq1}(\Omega_{2^{-k}\delta}\cap Q_{2^{-k}}) \supset \{(x,0): x\in B_{1-\delta}, \, {\rm dist}(x,\partial\Omega\cap\{t=0\})>2\delta|x|\},$$
    we deduce $u(\cdot,0) \geq 0$ in 
    $$\{(x,0): x\in B_{1-\delta}, \, {\rm dist}(x,\partial\Omega\cap\{t=0\})>2\delta|x|\}.$$
    
    Now (taking $\delta/2$ instead of $\delta$), we can repeat the argument for $$u\left(x_0+\frac{x}{2},t_0+\frac{t}{4}\right),$$
    for all $(x_0,t_0) \in \p\Omega\cap Q_{1/2}$, and therefore we find $u \geq 0$ in $\Omega\cap Q_{1/2}$.
\end{proof}

Then we check that our sets of interest are parabolically accessible.

\begin{lem}\label{lem:D_eta_accessible}
    Let $\eta \in (0,\frac{1}{4})$ and $D_{\eta,m}$ as in Definition \ref{defn:D_eta}.
    
    Then, $D_{\eta,m}$ is parabolically accesible in the sense of Definition \ref{defn:boundary_accessible}, with constants depending only on the dimension (and not on $\eta$).
\end{lem}

\begin{proof}
    We will verify Definition \ref{defn:boundary_accessible} with $N=4$, $\delta_0 = \frac{1}{2}$, and a dimensional $c_0 > 0$ to be chosen later.

    \noindent\textit{Step 1.} We start with a geometric observation: for any $p = (\bar x_0,\bar y_0, t_0) \in D_{\eta,m}$,
    $$\operatorname{d_p}(p,\p D_{\eta,m}) \leq \min\{|\bar y_0|-\eta|t_0|^{1/2},|\bar y_0|-\eta|\bar x_0|\} \leq (1+\eta)\operatorname{d_p}(p,\p D_{\eta,m}).$$
    Indeed, the first inequality follows from $\operatorname{d_p}$ being a distance. For the second one, first note that since $p \in D_{\eta,m}$,
    $$\operatorname{d_p}(p,\p D_{\eta,m}) = \min\left\{\operatorname{d_p}(p,\{|\bar y| = \eta|t|^{1/2}\}),\operatorname{d_p}(p,\{|\bar y| = \eta|\bar x|\})\right\}.$$
    Now, let $s := \dfrac{|\bar y_0| - \eta|t_0|^{1/2}}{1+\eta}$. Then,
    \begin{align*}
        \operatorname{d_p}(p,\{|\bar y| = \eta|t|^{1/2}\}) &\geq \inf\left\{|\bar y_0| - \eta|t|^{1/2}, \ \text{for} \ t \in \left[t_0 - s^2,\ t_0 + s^2\right]\right\}\\
        &=|\bar y_0|-\eta\sqrt{|t_0 - s^2|}\\
        &\geq |\bar y_0|-\eta\sqrt{|t_0|}-\eta s = \frac{|\bar y_0|-\eta|t_0|^{1/2}}{1+\eta}.
    \end{align*}
    On the other hand,
    $$\operatorname{d_p}(p,\{|\bar y| = \eta|\bar x|\}) = \frac{|\bar y_0|-\eta|\bar x_0|}{\sqrt{1+\eta^2}} \geq \frac{|\bar y_0|-\eta|\bar x_0|}{1+\eta}.$$
    
    \noindent\textit{Step 2.} We will now check condition (i) in Definition \ref{defn:boundary_accessible}. Let $\delta \in (0,\frac{1}{2})$, assume without loss of generality that $\bar y_0 = (\rho,0,\cdots,0)$, and define
    $$(\bar x_i,\bar y_i,t_i) := \left(\bar x_0, \rho + \frac{\delta}{4}i,0,\ldots,0,t_0-\frac{\delta^2}{16}i\right).$$
    Then, $r_i = \frac{\delta}{4}$ for all $i$, $\operatorname{d_p}((\bar x_0,\bar y_0,t_0),D_{\eta,m}) \geq \frac{\delta}{2}$ by assumption,
    \begin{align*}
        \operatorname{d_p}((x_i,t_i),D_{\eta,m}) &\geq \frac{1}{1+\eta}\min\left\{\rho+\frac{\delta}{4}i-\eta\sqrt{|t_0|+\frac{\delta^2}{16}i},\,\rho+\frac{\delta}{4}i-\eta|\bar x_0|\right\}\\
        &\geq \frac{1}{1+\eta}\min\left\{\rho-\eta\sqrt{|t_0|} + \frac{\delta}{4}(1-\eta)i,\,\frac{\delta}{2}+\frac{\delta}{4}i\right\}\\
        &\geq \frac{1}{1+\eta}\min\left\{\frac{\delta}{2}+\frac{\delta}{4}(1-\eta),\,\frac{3\delta}{4}\right\} \geq \frac{11\delta}{20}
    \end{align*}
    for all $i \geq 1$, and finally 
    $$\operatorname{d_p}((x_4,t_4),D_{\eta,m}) \geq \frac{1}{1+\eta}\min\left\{\frac{\delta}{2}+\delta(1-\eta),\,\frac{\delta}{2}+\delta\right\} \geq \delta.$$
    
    \noindent\textit{Step 3.} Finally, we check condition (ii). Given $(\bar x_0,\bar y_0, t_0) \in Q_{1-2\delta}\setminus D_{\eta,m}$, assume as before that $\bar y_0 = (\rho,0,\ldots,0)$.
    
    Let $E = B_{2\delta}((\bar x_0,\bar y_0))\cap\left\{ \bar y^{(1)} \geq \rho + \frac{7\delta}{4}\right\}$, where $\bar y^{(1)}$ represents the first coordinate of $\bar y$. Then, $E\times\{t_0-4\delta^2\} \subset \p_p Q_{2\delta}(x_0,t_0)$, and $|E| = c_0|\p_pQ_{2\delta}|$ for a dimensional $c_0 > 0$. Moreover,
    \begin{align*}
        \operatorname{d_p}(E\times\{t_0-4\delta^2\},D_{\eta,m}) &\geq \frac{1}{1+\eta}\min\left\{\rho+\frac{7\delta}{4}-\eta\sqrt{|t_0|+4\delta^2},\,\rho+\frac{7\delta}{4}-\eta|\bar x_0|\right\}\\
        &> \frac{1}{1+\eta}\min\left\{\rho-\eta\sqrt{|t_0|}+\frac{7\delta}{4}-2\eta\delta,\,\frac{7\delta}{4}\right\} > \delta.
    \end{align*}
    Thus, we have checked that our domain is parabolically accessible in the sense of Definition \ref{defn:boundary_accessible}.    
\end{proof}

Finally, thanks to the interior Harnack, we can give the:

\begin{proof}[Proof of Proposition \ref{cor:almost_positivity}]
	We divide the proof into two steps:
	
    \noindent\textit{Step 1.} We show that, for every $\delta \in (0,\frac{1}{2})$, there exists $N_\delta \in \N$ (not depending on $\eta$) such that, for every point $(\bar x_0,\bar y_0,t_0)$ in
    $$\Omega_\delta := \left\{(\bar x,\bar y,t) \in D_{\eta,m}\cap Q_{2/3} \ | \ \operatorname{d_p}((\bar x,\bar y,t),\p D_{\eta,m}) > \delta\right\},$$
    there exists a Harnack chain $(\bar x_i,\bar y_i,t_i)$, $i = 0,\ldots,N \leq N_\delta$, satisfying 
    
    $$Q_{2\sqrt{t_i-t_{i+1}}}(\bar x_i,\bar y_i,t_i) \subset D_{\eta,m}, \ |\bar x_i-\bar x_{i+1}|^2 + |\bar y_i - \bar y_{i+1}|^2 \leq |t_i-t_{i+1}|,$$
    and
    $$|\bar y_N| \in \left[\frac{3}{8},\frac{2}{3}\right], \ t_N \in [-0.99,0].$$
    
    To prove it, assume first without loss of generality that $\bar y_0 = (\rho,0,\ldots,0)$, and that $\rho < \frac{3}{8}$, otherwise it suffices to take $N = 0$.
    
    Now, we will take
    $$\left\{\begin{array}{rcl}
        \bar x_{i+1} &:= &\bar x_i,\\
        \bar y_{i+1} &:= &\left(|y_i|+\dfrac{1}{2}\operatorname{d_p}((\bar x_i,\bar y_i,t_i),D_{\eta,m}),0,\ldots,0\right),\\
        t_{i+1} &_= &t_i - \dfrac{1}{4}\operatorname{d_p}((\bar x_i,\bar y_i,t_i),D_{\eta,m})^2.
    \end{array}\right.$$
    By the computations in the proof of Lemma \ref{lem:D_eta_accessible} (Step 2),
    $$\operatorname{d_p}((\bar x_{i+1},\bar y_{i+1},t_{i+1}),D_{\eta,m}) \geq 1.1\operatorname{d_p}((\bar x_i,\bar y_i,t_i),D_{\eta,m}).$$
    Moreover, since $\operatorname{d_p}((\bar x,\bar y,t),D_{\eta,m}) \leq |\bar y|$ (cf. Lemma \ref{lem:D_eta_accessible}, Step 1),
    $$|\bar y_{i+1}| \leq 1.5|\bar y_i|.$$
    
    Then, we can define $N_\delta$ as the minimum positive integer such that $1.1^{N_\delta}\delta > \frac{3}{8}$, and then choose some $N \leq N_\delta$ such that $|\bar y_N| \in [\frac{3}{8},\frac{9}{16}]$.
    
    Finally, we estimate
    $$|t_N| = |t_0| + \sum\limits_{i = 1}^N|t_i-t_{i-1}| = |t_0| + \sum\limits_{i=1}^N|\bar y_i-\bar y_{i-1}|^2 \leq |t_0| + |\bar y_N- \bar y_0|^2 \leq \frac{2}{3}+\frac{81}{256} < 0.99.$$
    
    \noindent\textit{Step 2.} We finish the proof combining the interior Harnack with Proposition \ref{prop:almost_positivity_bulk}.
    
    First, by Theorem \ref{thm:interior_Harnack} applied to $u+\nu$, 
    $$u+\nu \geq c_1 > 0 \quad \text{in} \ \left\{\frac{3}{8} \leq |\bar y| \leq \frac{2}{3},\, -0.99 \leq t \leq 0\right\}\cap Q_1.$$
    Then, using the interior Harnack repeatedly on the Harnack chain constructed in Step 2, we deduce that $u+\nu \geq c^{N_\delta}c_1$ in $\Omega_\delta$.
    
    Now,
    $$v(\bar x,\bar y,t) := \nu^{-1/2}u\left(\frac{2}{3}\bar x,\frac{2}{3}\bar y,\frac{4}{9}t\right)$$
    satisfies
    $$\left\{\begin{array}{rclll}
        v_t-\Delta v & = & 0 & \text{in} & D_{\eta,m}\\
        v & \geq & \!-\sqrt{\nu} & \text{in} & D_{\eta,m}\cap Q_1\\
        v & \geq & \nu^{-1/2}(c^{N_\delta}c_1-\nu) & \text{in} & \Omega_{3\delta/2}\cap Q_1\\
        v & = & 0 & \text{on} & \p D_{\eta,m}\cap Q_1,
    \end{array}\right.$$
    and finally choosing $\delta$, and then $\nu$ small enough, we can apply Proposition \ref{prop:almost_positivity_bulk} to obtain $v \geq 0$ in $D_{\eta,m}\cap Q_{1/2}$. The result follows by a standard covering argument.
\end{proof}

\section{Self-similar solutions}\label{sec:homo}

In this section we will construct self-similar solutions to the heat equation in the complement of \textit{parabolic cones}, and estimate their growth rates. 
To do so, we follow the construction in \cite[Section A.4]{Tor24} (see also \cite[Lemma 5.8]{FRS24}).

\begin{prop}\label{prop:cone_solns}
    Let $m \in \{0,\ldots,n-2\}$. There exists $\eta_0 \in (0,\frac{1}{4})$, depending only on $n$ and $m$, such that for every $\eta \in (0,\eta_0)$, there exists a unique positive solution to
    $$\p_t\varphi_\eta - \Delta\varphi_\eta = 0 \quad \text{in} \quad D_{\eta,m}$$
    such that 
    \[\varphi_\eta(\lambda x,\lambda^2 t) = \lambda^{2\varepsilon}\varphi_\eta(x,t)\quad \textrm{for all}\quad \lambda>0,\]
     for some $\varepsilon > 0$, and $\varphi_\eta(e_n,-1) = 1$. Here $D_{\eta,m}$ is as in Definition \ref{defn:D_eta}.
    
    Moreover, $\|\varphi_\eta\|_{L^\infty(Q_1)} \leq C_*$, $\varphi_\eta \geq c_*$ on $\{|\bar y| = \frac{1}{2}\}\cap Q_1$, and
    \begin{itemize}
        \item[(a)] If $m = n-2$, we have $\varepsilon \leq C_*|\log\eta|^{-1}$.
        
        \vspace{0.1cm}
        \item[(b)] If $m < n-2$, we have $\varepsilon \leq C_*\eta^{n-m-2}$.
    \end{itemize}
    The constants $c_*$ and $C_*$ are positive, and they depend only on $m$ and $n$ (not on $\eta$).
\end{prop}

We recall the Gaussian log-Sobolev inequality, that will be needed in the proof.
\begin{lem}\label{lem:log_sobolev}
    Let $f \in H^1(\R^n;\mu)$, where $\mu$ is the Gaussian measure,
    $$\mathrm{d}\mu = (4\pi)^{-n/2}e^{-|x|^2/4}\mathrm{d}x.$$
    Then,
    $$\int f^2\log f^2\mathrm{d}\mu \leq \int |\nabla f|^2\mathrm{d}\mu + \left(\int f^2\mathrm{d}\mu\right)\log\left(\int f^2\mathrm{d}\mu\right).$$
\end{lem}

Now we will prove our characterization of self-similar solutions.
\begin{proof}[Proof of Proposition \ref{prop:cone_solns}]
    The proof is divided into four steps. First, we construct~$\varphi_\eta$. Then, we estimate $\varepsilon$. In the third step, we estimate $\|\varphi_\eta\|_{L^\infty(Q_1)}$, and in the last one we show that $\varphi_\eta \geq c > 0$ on $\{|\bar y| = \frac{1}{2}\}$.

    \noindent\textit{Step 1.} We construct $\varphi_\eta$.    
    We can write 
    \[\varphi_\eta(x,t) = c_\varepsilon|t|^\varepsilon\phi(x/|t|^{1/2}).\]
    Then, $\phi$ solves the following eigenvalue problem for the Ornstein-Uhlenbeck operator (see \cite[Lemma 5.8]{FRS24}):
    $$\left\{\begin{array}{rclll}
        \L_{OU}\phi + \varepsilon\phi & = & 0 & \text{in} & \R^n \setminus \left\{|\bar y| \leq \eta\max\{1,|x|\}\right\}\\[0.1cm]
        \phi & = & 0 & \text{on} & \p\left\{|\bar y| \leq \eta\max\{1,|x|\}\right\},
    \end{array}\right.$$
    where
    $$\L_{OU}\phi(x) := \Delta\phi(x) - \frac{x}{2}\cdot\nabla\phi(x) = e^{|x|^2/4}\operatorname{div}(e^{-|x|^2/4}\nabla\phi).$$   
    Since $\phi$ is positive, it is the first eigenfunction for $\L_{OU}$ in this domain, and therefore by the Rayleigh quotient characterization,
    $$\varepsilon = \inf\limits_{u \in C^{0,1}_c(\R^n \setminus \left\{|\bar y| \leq \eta\max\{1,|x|\}\right\}), \ \|u\|_{L^2_w} = 1}(4\pi)^{-n/2}\int|\nabla u|^2e^{-|x|^2/4},$$
    where 
    $$\|u\|_{L^2_w}^2 := (4\pi)^{-n/2}\int u^2e^{-|x|^2/4},$$
    and the infimum is attained by a unique function $\phi_\eta \in L^2_w$ by standard arguments. To obtain the desired normalization, we choose $c_\varepsilon = \phi_\eta(e_n)^{-1}$.
    
    \noindent\textit{Step 2.} Then, we estimate precisely $\varepsilon$ using a competitor. Let $f_{\eta,2}, f_{\eta,k} : [0,\infty) \to \R$ be defined as
    $$\left\{\begin{array}{rcl}
        f_{\eta,2}(r) & = & 1 + \dfrac{\log r}{|\log \eta|},\\[0.4cm]
        f_{\eta,k}(r) & = & \dfrac{r^{k-2}-\eta^{k-2}}{r^{k-2}(1-\eta^{k-2})},
    \end{array}\right.$$
    with $k \geq 3$. Then, let $u$ be defined on $\left\{|\bar y| > \eta\max\{1,|x|\}\right\}$ as
    $$u := f_{\eta,{n-m}}\left(\min\{1,|\bar y|\}\right)f_{\eta,n-m}\left(\frac{|\bar y|}{|x|}\right).$$   
    Then,
    $$\varepsilon \leq \left(\int u^2e^{-|x|^2/4}\right)^{-1}\int|\nabla u|^2e^{-|x|^2/4}.$$    
    First, since $\eta < \frac{1}{4}$,
    $$f_{\eta,2}\left(\frac{1}{2}\right) > f_{1/4,2}\left(\frac{1}{2}\right) = \frac{1}{2},$$
    and if $k \geq 3$,
    $$f_{\eta,k}\left(\frac{1}{2}\right) > f_{1/4,k}\left(\frac{1}{2}\right) = \frac{2^{2-k}-4^{2-k}}{2^{2-k}(1-4^{2-k})} = \frac{1 - 2^{2-k}}{1 - 4^{2-k}} > \frac{1}{2}.$$
    Then, it follows that if $|\bar y| \geq 1$ and $2|\bar y| \geq |x|$, $u > \frac{1}{2}$. Hence,
    $$\int u^2e^{-|x|^2/4} > \frac{1}{2}\int_{\{2|\bar y| \geq |x|\}\cap\{|\bar y| \geq 1\}}e^{-|x|^2/4} = c(n,m) > 0.$$    
    Now, to control the gradient, we compute
    \begin{align*}
        |\nabla u| &\leq f_{\eta,n-m}'\left(\min\{1,|\bar y|\}\right)\left|\nabla |\bar y|\right|\chi_{\{|\bar y| \leq 1\}}f_{\eta,n-m}\left(\frac{|\bar y|}{|x|}\right)\\
        &\quad+f_{\eta,n-m}'\left(\frac{|\bar y|}{|x|}\right)\left|\nabla\frac{|\bar y|}{|x|}\right|f_{\eta,n-m}\left(\min\{1,|\bar y|\}\right)\\
        &\leq f_{\eta,n-m}'\left(|\bar y|\right)\chi_{B_1}(\bar y) + f_{\eta,n-m}'\left(\frac{|\bar y|}{|x|}\right)\frac{|\bar x|}{|x|^2}\\
        &\leq f_{\eta,n-m}'\left(|\bar y|\right)\chi_{B_1}(\bar y) + \frac{1}{|x|}f_{\eta,n-m}'\left(\frac{|\bar y|}{|x|}\right).
    \end{align*}    
    Then, since $u$ is defined on $\left\{|\bar y| > \eta\max\{1,|x|\}\right\}$,    
    $$\int|\nabla u|^2e^{-|x|^2/4} \leq \int_{\{\eta < |\bar y| < 1\}}f'_{\eta,n-m}\left(|\bar y|\right)^2e^{-|x|^2/4} + \int_{\{|\bar y| > \eta|x|\}}f'_{\eta,n-m}\left(\frac{|\bar y|}{|x|}\right)^2\frac{e^{-|x|^2/4}}{|x|}.$$    
    On the one hand,
    \begin{align*}
        \int_{\{\eta < |\bar y| < 1\}}f'_{\eta,n-m}\left(|\bar y|\right)^2e^{-|x|^2/4} &= \int e^{-|\bar x|^2/4}\mathrm{d}\bar x\int_{\{\eta < |\bar y| < 1\}}f'_{\eta,n-m}\left(|\bar y|\right)^2e^{-|\bar y|^2/4}\mathrm{d}\bar y\\
        &\lesssim \int_\eta^1f'_{\eta,n-m}(r)^2r^{n-m-1}\mathrm{d}r,
    \end{align*}
    and thus if $m = n-2$,
    $$\int_{\{\eta < |\bar y| < 1\}}f'_{\eta,2}\left(|\bar y|\right)^2e^{-|x|^2/4} \lesssim \int_\eta^1\frac{r}{(r\log\eta)^2}\mathrm{d}r \lesssim |\log\eta|^{-1},$$
    and if $m < n-2$,
    $$\int_{\{\eta < |\bar y| < 1\}}f'_{\eta,n-m}\left(|\bar y|\right)^2e^{-|x|^2/4} \lesssim \int_\eta^1\left(\frac{\eta^{n-m-2}}{r^{n-m-1}}\right)^2r^{n-m-1}\mathrm{d}r \lesssim \eta^{n-m-2}.$$
    On the other hand, 
    \begin{align*}
        \int_{\{|\bar y| > \eta|x|\}}f'_{\eta,n-m}\left(\frac{|\bar y|}{|x|}\right)^2\frac{e^{-|x|^2/4}}{|x|} &= \int_{\p B_1\cap\{|\bar y| > \eta\}}f'_{\eta,n-m}\left(|\bar y|\right)^2\int_0^\infty\frac{e^{-\rho^2/4}}{\rho}\rho^{n-1}\mathrm{d}\rho\\
        &\lesssim \int_{\p B_1\cap\{|\bar y| > \eta\}}f'_{\eta,n-m}\left(|\bar y|\right)^2\\
        &= \int_\eta^1f'_{\eta,n-m}(r)^2\left|\p B_1\cap\{|\bar y| = r\}\right|\mathrm{d}r\\
        &= \int_\eta^1f'_{\eta,n-m}(r)^2\left|\{|\bar x| = \sqrt{1-r^2},\ |\bar y| = r\}\right|\mathrm{d}r\\
        &\lesssim \int_\eta^1f'_{\eta,n-m}(r)^2r^{n-m-1}\mathrm{d}r,
    \end{align*}
    and we can proceed as with the first term. Summing up, if $m = n-2$,
    $$\int|\nabla u|^2e^{-|x|^2/4} \lesssim |\log\eta|^{-1},$$
    and if $m < n-2$,
    $$\int|\nabla u|^2e^{-|x|^2/4} \lesssim \eta^{n-m-2},$$
    which combined with
    $$\int u^2e^{-|x|^2/4} \geq c(n,m) > 0$$
    yields the desired result.
    
    \noindent\textit{Step 3.} Then, we prove that $\|\varphi_\eta\|_{L^\infty(Q_1)} \leq C$, independently of $\eta$.
    
    Recall that in Step 1 we defined $\varphi_\eta = c_\varepsilon^{-1}\phi_\eta$. The upper bound comes from an application of the interior Harnack and the Gaussian log-Sobolev inequality. First,
    note that $t^2 \leq 1 + t^2\log t^2$, and then
    $$(4\pi)^{-n/2}\int (\phi_\eta^2-1)_+e^{-|x|^2/4} \leq (4\pi)^{-n/2}\int \phi_\eta^2\log_+\big(\phi_\eta^2\big) e^{-|x|^2/4}.$$
    
    Now, by the Gaussian log-Sobolev inequality (Lemma \ref{lem:log_sobolev}) applied to $\phi_\eta\chi_{\{\phi_\eta > 1\}}$,
    $$(4\pi)^{-n/2}\int(\phi_\eta^2-1)_+e^{-|x|^2/4} \leq (4\pi)^{-n/2}\int|\nabla\phi_\eta|^2e^{-|x|^2/4} = \varepsilon.$$   
    Now, let
    $$A_1 = B_1 \cap \left\{|\bar y| \geq \frac{1}{2}\max\left\{1,|x|\right\}\right\}.$$
    Then,
    \begin{align*}
        (4\pi)^{-n/2}\int_{A_1}\phi_\eta^2e^{-|x|^2/4} &= 1 - (4\pi)^{-n/2}\int_{\R^n\setminus A_1}\phi_\eta^2e^{-|x|^2/4}\\
        &= (4\pi)^{-n/2}\int_{\R^n}e^{-|x|^2/4} - (4\pi)^{-n/2}\int_{\R^n\setminus A_1}\phi_\eta^2e^{-|x|^2/4}\\
        &\geq (4\pi)^{-n/2}\int_{A_1}e^{-|x|^2/4} - (4\pi)^{-n/2}\int_{\R^n\setminus A_1}(\phi_\eta^2-1)_+e^{-|x|^2/4}\\
        &\geq (4\pi)^{-n/2}\int_{A_1}e^{-|x|^2/4} - \varepsilon \geq a - \varepsilon \geq \frac{a}{2},
    \end{align*}
    provided that $\eta_0$ (and then $\varepsilon$) is small enough. Then,
    $$\sup\limits_{A_1}\phi_\eta \geq \left(\frac{\int_{A_1}\phi_\eta^2 e^{-|x|^2/4}}{\int_{A_1} e^{-|x|^2/4}}\right)^{1/2} \geq \frac{1}{\sqrt{2}}.$$
    
    Now, $\tilde\varphi := |t|^\varepsilon\phi(x/|t|^{1/2})$ is a solution to the heat equation in $D_{\eta,m}$, and in particular in $Q_{2} \cap \{|\bar y| \geq \frac{1}{2}\}$. Now let $$A_{\sqrt{2}} := B_{\sqrt{2}}\cap\left\{|\bar y| \geq \max\left\{1,|x|\right\}\right\}.$$
    By the interior Harnack,
    $$\phi_\eta(e_n) \geq \inf\limits_{A_{\sqrt{2}}\times\{-1\}} \tilde\varphi \geq c \sup\limits_{A_{\sqrt{2}}\times\{-\sqrt{2}\}} \tilde\varphi \geq 2^{\varepsilon/2}c\sup\limits_{A_1}\phi_\eta \geq \frac{c}{2},$$
    and hence the constant $c_\varepsilon = \phi_\eta(e_n)^{-1}$ from Step 1 is uniformly bounded as $\eta \rightarrow 0^+$.
    
    To end the argument, we just use that $\varphi_\eta$ is a subsolution to the heat equation in the full space, exactly as in \cite[Proposition 6.3]{Tor24}. 
    Then, for every $(x,t) \in Q_{1/2}$,
    \begin{align*}
        \varphi_\eta(x,t) &\leq C\int c_\varepsilon\phi_\eta(y)e^{-\frac{|x-y|^2}{4(1+t)}}\mathrm{d}y \leq C\int c_\varepsilon\phi_\eta(y)e^{-|x-y|^2/4}\\
        &\leq Cc_\varepsilon\left(\int e^{-|x-y|^2/6}\right)^{1/2}\left(\int\phi_\eta^2e^{-|x-y|^2/3}\right)^{1/2} \leq Cc_\varepsilon\left(\int\phi_\eta^2e^{-|y|^2/4}\right) \leq C,
    \end{align*}
    where we used that for all $x \in B_{1/2}$,
    $$-\frac{|x-y|^2}{3} \leq C - \frac{|y|^2}{4}.$$
    Hence, by homogeneity, $\|\varphi_\eta\|_{L^\infty(Q_1)} \leq 2^\varepsilon C$.
    
    \noindent\textit{Step 4.} We finally show that $\varphi_\eta \geq c > 0$ on $\{|\bar y| = \frac{1}{2}\}$, independently of $\eta$.
    
    For the lower bound, we start observing that by symmetry, $\varphi_\eta \equiv 1$ on $$E_1 = \{|\bar x| = 0, |\bar y| = 1, t = -1\},$$
    and by homogeneity, $\varphi_\eta \equiv 2^{\varepsilon/4} > 1$ on
    $$E_{\sqrt{2}} = \{|\bar x| = 0, |\bar y| = \sqrt{2}, t = -\sqrt[4]{2}\}.$$
    
    Now, $\varphi_\eta$ is a solution to the heat equation in $Q_{3/2}\cap\{|\bar y| > \frac{3}{8}\}$. Hence, by the interior Harnack inequality,
    $$\inf\limits_{Q_1\cap\{|\bar y| = \frac{1}{2}\}} \varphi_\eta \geq c\sup\limits_{E_{\sqrt{2}}}\varphi_\eta > c,$$
    as we wanted to prove.
\end{proof}

\section{Extinction rate at singular points}\label{sec:rates}

\subsection{Nondegeneracy}

The expansion in Theorem \ref{thm:expansion} gives rise to an upper bound on the contact set, that is, at small scales, we can ensure that $u$ is positive away from the zero set of the blow-up $p_2$.

\begin{prop}\label{prop:contact_set_ub}
    Let $u$ be a solution to \eqref{eq:obstacleProblem} such that $(0,0)\in\Sigma_m$. Then (after a rotation), there exists $C > 0$ such that for all $r \in (0,\frac{1}{2}]$,
    $$\{u > 0\}\cap Q_r \supset D_{C\sqrt{\sigma(r)},m},$$
    where $D_{\eta,m}$ is as in Definition \ref{defn:D_eta}, and $\sigma(r)$ is as defined in Theorem \ref{thm:expansion}.
    
    In particular, we have 
    $$\{u = 0\}\cap Q_r \subset \left\{|\bar y|^2 < C_1\sigma(r)|x|^2\right\} \cup \left\{t< -C_1^{-1}|\bar y|^2/\sigma(r)\right\}$$
    for some constant $C_1$.
\end{prop}

\begin{proof}
    From Theorem \ref{thm:expansion}, we have that
    $$u(x,t) \geq p_2(x) - C(|x|^2+|t|)\sigma(r) \quad \text{in} \ Q_r.$$
    Now, by Proposition \ref{pro:classification}, we can write
    $$p_2(x) = \frac{1}{2}\sum\limits_{i = m+1}^n\mu_ix_i^2 \geq c|\bar y|^2.$$
    It follows that
    $$u(x,t) \geq c|\bar y|^2 - C\sigma(r)(|x|^2+|t|) \quad \text{in} \ Q_r,$$
    and then
    \begin{align*}
        \{u > 0\}\cap Q_r &\supset \left\{c|\bar y|^2 > C\sigma(r)(|x|^2+|t|)\right\}\\
        &\supset \left\{|\bar y| > \sqrt{\frac{2C\sigma(r)}{c}}|x|\right\}\cap\left\{|\bar y| > \sqrt{\frac{2C\sigma(r)}{c}|t|}\right\}.
    \end{align*}
    The result follows by the definition of $D_{\eta,m}$.
\end{proof}

\subsection{Regularity}
Comparing the set where $u$ is positive with self-similar domains, we obtain a very precise lower bound for $u_t$ in terms of the homogeneous solutions defined in Proposition \ref{prop:cone_solns}.
\begin{lem}\label{lem:growth_dyadic}
    Let $u$ be a solution to \eqref{eq:obstacleProblem} such that $(0,0)\in\Sigma_m$. Then, there exist $r_0, c_0, c > 0$ such that
    \[u_t \geq c_0c^k\cdot 2^{-\sum\limits_{j=0}^{k-1}2^{j}\varepsilon_j}\varphi_k \quad \text{in} \ Q_{r_k},\]
    where $r_k := 2^{1-2^k}r_0$, $\varphi_k$ and $\varepsilon_k$ are as defined in Proposition \ref{prop:cone_solns} with $\eta_k = 2C\sqrt{\sigma(2r_k)}$, $C$ comes from Proposition \ref{prop:contact_set_ub}, and $\sigma$ from Theorem \ref{thm:expansion}.
\end{lem}

\begin{proof}
    Let $r_0 \in (0,\frac{1}{4})$ small enough so that $C\sqrt{\sigma(2r_0)} < \frac{1}{8}$. Let $r_k := 2^{1-2^k}r_0$, $\eta_k := 2C\sqrt{\sigma(2r_k)}$, and let $\varphi_k$ and $\varepsilon_k$ be as defined in Proposition \ref{prop:cone_solns}. We will prove that
    $$u_t \geq c_k\varphi_k \quad \ \text{in} \ Q_{r_k},$$
    where $c_{k+1} = c(2r_{k+1}^{2(\varepsilon_k-\varepsilon_{k+1})})c_k$.
    
    First, note that $\varphi_0$ is a solution to the heat equation in $D_{\eta_0,m}\cap Q_{r_0}$. Then we can write
    $$\p_p\left(D_{\eta_0,m}\cap Q_{r_0}\right) = \left(\p D_{\eta_0,m}\cap Q_{r_0}\right) \cup \left(\p_p Q_{r_0}\cap D_{\eta_0,m}\right) =: I\cup II.$$
    Then, using Proposition \ref{prop:contact_set_ub} we have that on $I$, $u_t > 0$ and $\varphi_0 \equiv 0$, while on $II$, $u_t$ and $\varphi_0$ are continuous and positive. Since $II$ is at a positive distance from the boundary of $\{u > 0\}$, by compactness $u_t \geq \mu > 0$ on $II$, and it follows that $u_t \geq c_0\varphi_0$ on $II$ for some positive $c_0$. Therefore, by the comparison principle, $u_t \geq c_0\varphi_0$ in $Q_{r_0}$.
    
    Now we proceed with an iteration scheme. Assume by induction hypothesis that $u_t \geq c_k\varphi_k$ in $Q_{r_k}$. By Proposition \ref{prop:contact_set_ub} again, $u_t > 0$ in $D_{\eta_{k+1},m}\cap Q_{2r_{k+1}}$.
    
    Let
    $$v(x,t) := (1+\nu)c_k^{-1}c_*^{-1}(2r_{k+1})^{-2\varepsilon_k}u_t\left(2r_{k+1}x,4r_{k+1}^2t\right) - \nu\frac{\varphi_{k+1}}{\|\varphi_{k+1}\|_{L^\infty(Q_1)}},$$
    where $\nu$ comes from Proposition \ref{cor:almost_positivity}, and $c_*$ is the one in Proposition \ref{prop:cone_solns}. Then, $v \geq -\nu$ in $D_{\eta_{k+1},m}\cap Q_1$, it is nonnegative on $\p D_{\eta_{k+1},m}$, and
    $$v \geq (1+\nu)c_k^{-1}c_*^{-1}(2r_{k+1})^{-2\varepsilon_k}c_k\varphi_k\left(2r_{k+1}x,4r_{k+1}^2t\right) - \nu \geq 1 \quad \text{on} \ \left\{|\bar y| = \frac{1}{2}\right\}\cap Q_1.$$
    Therefore, by Proposition \ref{cor:almost_positivity}, $v \geq 0$ in $Q_{1/2}$, and undoing the scaling
    $$u_t \geq \frac{\nu c_*}{(1+\nu)C_*}(2r_{k+1})^{2(\varepsilon_k-\varepsilon_{k+1})}c_k\varphi_{k+1} \quad \text{in} \ Q_{r_{k+1}},$$
    with $C_*$ from Proposition \ref{prop:cone_solns} too.
    
    Finally we compute  
    \begin{align*}
        c_k &= c_0\prod\limits_{j=0}^{k-1} c(2r_j)^{2(\varepsilon_j-\varepsilon_{j+1})} = c_0(4r_0)^{2(\varepsilon_0-\varepsilon_k)}c^k\prod\limits_{j=0}^{k-1}2^{-2^{j+1}(\varepsilon_j-\varepsilon_{j+1})}\\
        &= c_0(4r_0)^{2(\varepsilon_0-\varepsilon_k)}c^k2^{2^k\varepsilon_k - \varepsilon_0-\sum\limits_{j=0}^{k-1}2^j\varepsilon_j}\\
        &\geq c_02^{-\varepsilon_0}(4r_0)^{2\varepsilon_0}c^k2^{-\sum\limits_{j=0}^{k-1}2^j\varepsilon_j}.
    \end{align*}
\end{proof}

Now, using the information that we have on $\sigma(r)$, and the dependence of $\varepsilon_j$ on $r$, we can simplify the bound in Lemma \ref{lem:growth_dyadic}.
\begin{lem}\label{lem:growth_r}
    Let $u$ be a solution to \eqref{eq:obstacleProblem} such that $(0,0)\in\Sigma_m$. Then, there exist $r_1, c_1 > 0$ such that for all $r \in (0,r_1)$,
    $$u_t \geq c_1\tau(r) \quad \text{in} \ \left\{|\bar y|^2 \geq \frac{r^2}{16n}\right\}\cap Q_r,$$
    where $\tau : \R_+ \to \R_+$ is defined by
    \begin{equation}\label{eq:tau}
        \tau(r)=\left\{\begin{array}{lllll}
            \exp -|\log r|^{\frac{1}{2}+\alpha} & \text{if} & n = 2 & \text{and} & m = 0,\\[0.1cm]
            \exp -|\log r|^\alpha & \text{if} & n \geq 3 & \text{and} & m = 0,\\[0.1cm]
            \exp -C\frac{|\log r|}{\log|\log r|} & \text{if} & n \geq 3 & \text{and} & m = n-2,\\[0.2cm]
            \exp -|\log r|^{1-\theta} & \text{if} & n \geq 4 & \text{and} & m \in \{1,\ldots,n-3\},
        \end{array}\right.
    \end{equation}
    for any $\alpha > 0$ and some $C, \theta > 0$.
\end{lem}

\begin{proof}
    First, let $r_1 \leq r_0$ from Lemma \ref{lem:growth_dyadic}, small enough so that $u > 0$ in\linebreak $Q_{2r_1}\cap D_{1/(16\sqrt{n}),m}$ (see Proposition \ref{prop:contact_set_ub}). Now, for all $r \in (0,r_1)$, $u > 0$ (and hence $u$ and $u_t$ are solutions to the heat equation) in $D_{1/(16\sqrt{n}),m}\cap Q_{2r}$, and in particular in $\{|\bar y|^2 \geq \frac{r^2}{32n}\}\cap Q_{2r}$.

    Then, by the interior Harnack,
    $$\inf\left\{u_t(x,t) : (x,t) \in Q_r \text{ and } |\bar y|^2\geq\frac{r^2}{16n}\right\} \geq cu_t(\sqrt{2}re_n,-2r^2),$$
    and hence it suffices to prove that
    $$u_t(\sqrt{2}re_n,-2r^2) \geq c\tau(r).$$

    Now, given $r \in [2^{-2^{k+1}}r_1,2^{-2^k}r_1]$, by Lemma \ref{lem:growth_dyadic},
    $$u_t(\sqrt{2}re_n,-2r^2) \geq c_0c^k 2^{-\sum\limits_{j=0}^{k-1}2^j\varepsilon_j}\varphi_k(\sqrt{2}re_n,-2r^2) = c_0c^k 2^{-\sum\limits_{j=0}^{k-1}2^j\varepsilon_j}(\sqrt{2}r)^{2\varepsilon_k},$$
    where $\varphi_k$ and $\varepsilon_k$ are as in Proposition \ref{prop:cone_solns} with $\eta_k = 2C\sqrt{\sigma(2^{2-2^k}r_1)}$, and $\sigma$ comes from Theorem \ref{thm:expansion}. Furthermore, since $|\log r| \gtrsim 2^k$, we can estimate $c^k \gtrsim |\log r|^{-a}$ for some fixed $a > 0$, and then
    $$u_t(\sqrt{2}re_n,-2r^2) \gtrsim c_0|\log r|^{-a}\cdot2^{-\sum\limits_{j=0}^{k-1}2^j\varepsilon_j}(\sqrt{2}r)^{2\varepsilon_k}.$$
    
    We distinguish four cases:
    \begin{itemize}
        \item When $n = 2$ and $m = 0$, $\sigma(t) = 2^{-|\log t|^\gamma}$ for any $\gamma \in (0,\frac{1}{2})$. Then,
        $$\eta_k = 2C\cdot 2^{-\frac{1}{2}(|2^k-2|\log 2 + |\log r_1|)^\gamma},$$
        and hence, for sufficiently large $k$,
        $$\varepsilon_k \leq \frac{C_*}{|\log \eta_k|} \lesssim 2^{-k\gamma} \lesssim |\log r|^{-\gamma},$$
        and then
        $$\sum\limits_{j=0}^{k-1}2^j\varepsilon_j \lesssim 2^{k(1-\gamma)} \lesssim |\log r|^{1-\gamma}.$$
        All in all, for small $r$,
        $$u_t(\sqrt{2}re_n,-2r^2) \gtrsim c_0|\log r|^{-a}2^{-C|\log r|^{1-\gamma}}(\sqrt{2}r)^{C|\log r|^{-\gamma}},$$
        that is,
        \begin{align*}
            \log u_t(\sqrt{2}re_n,-2r^2) &\geq -C + \log c_0 - a\log|\log r| - C\log 2\cdot|\log r|^{1-\gamma}\\
            &\quad - C|\log \sqrt{2}r|\cdot|\log r|^{-\gamma}\\
            &\geq -C|\log r|^{1-\gamma} \geq -|\log r|^{1-\gamma'},
        \end{align*}
        for all $\gamma' < \gamma$, provided that $r$ is small enough.
        \item When $n \geq 3$ and $m = 0$, $\sigma(t) = |\log t|^{-\delta}$ for any $\delta \in (0,\frac{2}{n-2})$\footnote{We change the notation to $\delta$ to avoid shadowing the variables $\varepsilon_k$, already in use.}. Then,
        $$\eta_k = 2C\cdot (|2^k-2|\log 2 + |\log r_1|)^{-\delta/2},$$
        and hence, for sufficiently large $k$,
        $$\varepsilon_k \leq C_*\eta_k^{n-2} \lesssim 2^{-k\delta(n-2)/2} =: 2^{-k(1-\alpha)} \lesssim |\log r|^{-(1-\alpha)},$$
        where $\alpha \in (0,1)$ is arbitrarily small, and then by the same reasoning as before
        $$\log u_t(\sqrt{2}re_n,-2r^2) \geq -|\log r|^{\alpha'}$$
        for all $\alpha' > \alpha$.
        \item When $n \geq 3$ and $m = n - 2$, $\sigma(t) = |\log t|^{-\delta_0}$, for some $\delta_0 > 0$\footnote{This $\delta_0$ corresponds to $\varepsilon_0$ in \eqref{eq:semiconvexity}.}. Then,
        $$\eta_k = 2C\cdot (|2^k-2|\log 2 + |\log r_1|)^{-\delta_0/2},$$
        and hence, for sufficiently large $k$,
        $$\varepsilon_k \leq \frac{C_*}{|\log \eta_k|} \lesssim \frac{1}{k} \lesssim \frac{1}{\log|\log r|},$$
        and then
        $$\sum\limits_{j=0}^{k-1}2^j\varepsilon_j \lesssim \frac{2^k}{k} \lesssim \frac{|\log r|}{\log|\log r|}.$$
        Combining the estimates, for small enough $r$,
        $$u_t(\sqrt{2}re_n,-2r^2) \gtrsim c_0|\log r|^{-a}2^{-C\frac{|\log r|}{\log|\log r|}}(\sqrt{2}r)^{\frac{C}{\log|\log r|}},$$
        and thus
        \begin{align*}
            \log u_t(\sqrt{2}re_n,-2r^2) &\geq \log c_0 - a\log|\log r| - C\frac{|\log r|}{\log|\log r|}\\
            &\quad -C\frac{|\log \sqrt{2}r|}{\log|\log r|}\\
            &\geq -C\frac{|\log r|}{\log|\log r|}.
        \end{align*}
        \item When $n \geq 4$ and $m \in \{1,\ldots,n-3\}$, $\sigma(t) = |\log t|^{-\delta_0}$ and $\eta_k \lesssim 2^{-k\delta_0/2}$. Now,
        $$\varepsilon_k \leq C_*\eta_k^{n-m-2} \lesssim 2^{-k\delta_0(n-m-2)/2} =: 2^{-k \theta} \leq |\log r|^{-\theta},$$
        and then
        $$\log u_t(\sqrt{2}re_n,-2r^2) \geq -|\log r|^{1-\theta'}$$
        for all $\theta' < \theta$.
    \end{itemize}
\end{proof}

We can translate the lower bound on $u_t$ to an upper bound on $\frac{\nabla u}{u_t}$ up to the free boundary.
\begin{prop}\label{prop:Clara-Teo}
    Let $u$ be a solution to \eqref{eq:obstacleProblem} such that $(0,0)\in\Sigma_m$. Then, there exist $r_1, C_1 > 0$ such that for all $r \in (0,r_1)$,
    $$\left\|\frac{\nabla u}{u_t}\right\|_{L^\infty(Q_r)} \leq C_1\frac{r}{\tau(2r)},$$
    where $\tau(r)$ is as in \eqref{eq:tau}.
\end{prop}

\begin{proof}
    First, note that the statement is equivalent to proving
    $$C_1u_t \pm \frac{\tau(2r)}{r}u_i \geq 0 \quad \text{in} \ \{u > 0\}\cap Q_r$$
    for all spatial derivatives $u_i$. Note also that by the convergence of $u$ to the blow-up $p_2$,
    $$u \leq p_2(x)+\frac{r^2}{24n} \leq \frac{|\bar y|^2}{2}+\frac{r^2}{24n} \quad \text{in} \ Q_{2r},$$
    and
    $$|\nabla u| \leq |\nabla p_2|+r \leq \sum\limits_{i=m+1}^n\mu_i|x_i|+r \leq |\bar y|+r \quad \text{in} \ Q_{2r}$$
    
    Now, let $(x_0,t_0) \in Q_r$, and define
    $$v := C_1u_t \pm \frac{\tau(2r)}{r}u_i + 9n\frac{\tau(2r)}{r^2}\left(\frac{|x-x_0|^2+t_0-t}{2n+1}-u\right),$$
    which is a solution to the heat equation in $\{u > 0\}$. Then, 
    $$v(x_0,t_0) \geq \inf\limits_{\p_p (\{u > 0\}\cap Q_r(x_0,t_0))} v,$$
    and it suffices to check that $v \geq 0$ on the parabolic boundary of $\{u > 0\}\cap Q_r(x_0,t_0)$. We distinguish three regions:
    \begin{itemize}
        \item On $\p\{u > 0\}$, $u = |\nabla u| = 0$, and then
        $$v = \frac{\tau(2r)}{r^2}\frac{9n(|x-x_0|^2+t_0-t)}{2n+1} \geq 0.$$
        
        \item On $\p_p Q_r(x_0,t_0)\cap\left\{|\bar y|^2 \leq \frac{r^2}{4n}\right\}$,
        \begin{align*}
            v &\geq \tau(2r)\left(-\sqrt{\frac{1}{4n}}-1 + 9n\left(\frac{1}{2n+1}-\frac{1}{8n}-\frac{1}{24n}\right)\right)\\
            &\geq \tau(2r)\left(-\sqrt{\frac{1}{4n}}-1+9n\frac{1}{6n}\right) \geq 0.
        \end{align*}
        
        \item On $\p_p Q_r(x_0,t_0)\cap\left\{|\bar y|^2 \geq \frac{r^2}{4n}\right\} \subset \left\{|\bar y|^2 \geq \frac{(2r)^2}{16n}\right\}\cap Q_{2r}$, $u_t \geq c_1\tau(2r)$ by Lemma \ref{lem:growth_r}, and then
        \begin{align*}
            v &\geq C_1c_1\tau(2r) + \tau(2r)\left(-2-1+9n\left(\frac{1}{2n+1}-2-\frac{1}{24n}\right)\right)\\
            &\geq C_1c_1\tau(2r) + \tau(2r)(-3+3-18n-1) \geq 0,
        \end{align*}
        choosing $C_1 = 19nc_1^{-1}$.
    \end{itemize}
\end{proof}

Finally, from the estimation we can deduce free boundary regularity.

\begin{cor}\label{cor:FB_almost_C2}
Let $u$ be a solution to \eqref{eq:obstacleProblem} such that $(0,0) \in \Sigma_m$. Then, there exist $r_1, c_1 > 0$ such that
$$\{u = 0\}\cap Q_{r_1} \supset \left\{|x| \leq c_1\sqrt{|t|}\,\tau(|t|) \right\},$$
where $\tau(r)$ is as in \eqref{eq:tau}.
\end{cor}

\begin{proof}
    Let $r \in (0,r_1)$ with $r_1$ from Proposition \ref{prop:Clara-Teo}. Let $(x,t) \in Q_r$ such that $u(x,t) > 0$. Then, we can write
    $$0 < u(x,t) = \int_0^1\left(\frac{\p}{\p\lambda}u(\lambda x,\lambda t)\right)\mathrm{d}\lambda = \int_0^1(tu_t+x\cdot\nabla u)(\lambda x,\lambda t)\mathrm{d}\lambda.$$
    Then,
    $$0 < \int_0^1(tu_t+|x||\nabla u|)(\lambda x,\lambda t)\mathrm{d}\lambda,$$
    and in particular $tu_t(\lambda_\circ x,\lambda_\circ t) + |x||\nabla u(\lambda_\circ x,\lambda_\circ t)| > 0$ at some point $(\lambda_\circ x,\lambda_\circ t)\in Q_r$. Then, by Proposition \ref{prop:Clara-Teo},
    $$t > -|x|\frac{|\nabla u(\lambda_\circ x,\lambda_\circ t)|}{u_t(\lambda_\circ x,\lambda_\circ t)} \geq -C_1\frac{r|x|}{\tau(2r)},$$
    and hence we have proved that
$$\{u = 0\}\cap Q_r \supset \left\{t \leq -C_1\frac{r}{\tau(2r)}|x|\right\}$$
for all $r\in (0,r_1)$. In particular, $u = 0$ in
$$\{|x| \leq C_1^{-1}\tau(2r)r, \, t = -r^2\},$$
that is,
$$\{u = 0\}\cap Q_{r_1} \supset \left\{|x| \leq C_1^{-1}\sqrt{|t|}\tau(2\sqrt{|t|})\right\}.$$
To conclude, note that we can replace $\tau(2\sqrt{|t|})$ by $\tau(|t|)$ by adjusting the constants in the definition of $\tau$.
\end{proof}

\subsection{Proof of the main results}

Our main results: Theorems \ref{main-thm1}, \ref{main-thm2}, and \ref{main-thm3}, and Proposition \ref{prop-thm1}, follow from the regularity and nondegeneracy estimates in this Section. We provide their proofs for completeness.

\begin{proof}[Proof of Theorem \ref{main-thm1}]
	It follows from combining Proposition \ref{prop:contact_set_ub} and Corollary \ref{cor:FB_almost_C2}.
\end{proof}

The proof of Proposition \ref{prop-thm1} is based on the same strategy, but making adjustments to leverage the improved estimate on $u_t$.
\begin{proof}[Proof of Proposition \ref{prop-thm1}]
	Recall that $n = 2$ and $(0,0) \in \Sigma_0$. Now, using that\linebreak $u_t \leq  C e^{-C|\log r|^{\frac12}}$, following the proof of Theorem \ref{thm:expansion} gives
	$$|u(x,t) - p_2(x)| \leq C(|x|^2+|t|)\tilde\sigma(\sqrt{|x|^2+|t|}) \quad \text{in} \ Q_{1/2},$$
	where $\tilde\sigma(r) = Ce^{-C|\log r|^{1/2}}$. Then, an analogous computation to the proof of Proposition \ref{prop:contact_set_ub} gives that
	\[\partial\{u(\cdot,-t)>0\} \subset \left\{ |x|<C_1\sqrt{t}\,e^{-c_1|\log t|^{\frac12}} \right\}.\]
	
	Moreover, we can replace $\sigma$ by $\tilde\sigma$ in Lemma \ref{lem:growth_dyadic}, and hence replace $\tau(r)$ by\linebreak $\tilde\tau(r) = \exp -C|\log r|^{1/2}$, and then Corollary \ref{cor:FB_almost_C2} gives	
	\[\partial\{u(\cdot,-t)>0\} \subset \left\{ c_1\sqrt{t}\,e^{-C_1|\log t|^{\frac12}}<|x|\right\}.\]
\end{proof}

The proof of Theorem \ref{main-thm2} is again straightforward.
\begin{proof}[Proof of Theorem \ref{main-thm2}]
	It follows from combining Proposition \ref{prop:contact_set_ub} and Corollary \ref{cor:FB_almost_C2}.
\end{proof}

Finally, Theorem \ref{main-thm3} follows from combining the classical results for regular points in \cite{Caf77,KN77,KN78} with our results on extinction rates for singular points.
\begin{proof}[Proof of Theorem \ref{main-thm3}]
	 Let $(x_0,t_0) \in U$. If $(x_0,t_0)$ is a regular point, then $\p\{u > 0\}$ is locally a $C^\infty$ graph around $(x_0,t_0)$.
	 
	 Otherwise, $(x_0,t_0) \in \Sigma_m$ for some $m \in \{0,\ldots,n-2\}$. Now, by Proposition \ref{prop:Clara-Teo}, there exists $r_0 > 0$ such that
	 $$\left\|\frac{\nabla u}{u_t}\right\|_{L^\infty(Q_r(x,t))} \leq C_1\frac{r}{\tau(2r)} \leq C\sqrt{r},$$
	 for all $r \in (0,r_0)$. 
	 
	 This implies that, for any $\lambda > 0$ and $r \in (0,r_0)$,
	 $$\{u(x,t) = \lambda\}\cap Q_r = \{t = \Gamma_\lambda(x), \ x \in B_r(x_0)\},$$
	 where $\Gamma_\lambda : B_r(x_0) \to \R$ is a Lipschitz function with Lipschitz constant $C\sqrt{r}$. 
	 
	 Taking the limit $\lambda \downarrow 0$, we conclude that the free boundary is a $C\sqrt{r}$-Lipschitz graph in $Q_r(x_0,t_0)$. Moreover, if we let $r \downarrow 0$, we deduce that $\p\{u > 0\}$ has a \textit{horizontal} tangent plane at $(x_0,t_0)$.
	 
	 Finally, once we know the tangent plane is well defined at every free boundary point in $U$, note that the normal vector $\nu_{(x,t)}$ is continuous in the regular set,\linebreak $\nu_{(x,t)} \equiv e_{n+1}$ on the singular set, and, for every singular point $(x_0,t_0)$,
	 $$|\nu_{(x,t)} - e_{n+1}| \leq C\sqrt{r} \quad \text{in} \ Q_r(x_0,t_0),$$
	 which together gives that $\p\{u > 0\}$ is locally a $C^1$ graph.
\end{proof}


\begin{thebibliography}{FRS24}


\bibitem[AHV01]{AHV01} D. Andreucci, M. A. Herrero, J. J. Vel\'azquez, \textit{The classical one-phase Stefan problem: A catalogue on interface behaviours}, Surveys in Industrial Mathematics \textbf{9} (2001), 247–337.

\bibitem[ACS96]{ACS96} I. Athanasopoulos, L. Caffarelli, S. Salsa, \textit{Regularity of the free boundary in parabolic phase-transition problems}, Acta Math. \textbf{176} (1996), 245–282.

\bibitem[Bla06]{Bla06} A. Blanchet, \textit{On the singular set of the parabolic obstacle problem}, J. Differential Equations \textbf{231} (2006), 656-672.
    
    \bibitem[Caf77]{Caf77}
    L. Caffarelli, \textit{The regularity of free boundaries in higher dimensions}, Acta Math. \textbf{139} (1977), 155-184.
    
    \bibitem[Caf78]{Caf78} L. Caffarelli, \textit{Some aspects of the one-phase Stefan problem}, Indiana Math. J. \textbf{27} (1978), 73-77.
    
    \bibitem[CF78]{CF78} L. Caffarelli, A. Friedman, \textit{The one-phase Stefan problem and the porous medium diffusion equation: Continuity of the solution in $n$ space dimensions}, Proc. Nat. Acad. Sci. USA \textbf{75} (1978), no. 5, 2084.
    
    \bibitem[CF79]{CF79}
    L. Caffarelli, A. Friedman, \textit{Continuity of the temperature in the Stefan problem}, Indiana U. Math. J. \textbf{28} (1979), 53-70.

\bibitem[CK10]{CK10} S. Choi, I. Kim, \textit{Regularity of one-phase Stefan problem near Lipschitz initial data}, Amer. J. Math. \textbf{132} (2010), 1693–1727. 
    
    \bibitem[DS20]{DS20} D. De Silva, O. Savin, \textit{A short proof of boundary Harnack inequality}, J. Differential Equations \textbf{269} (2020), 2419-2429.
    
    \bibitem[DS22]{DS22} D. De Silva, O. Savin, \textit{On the parabolic boundary Harnack principle}, La Matematica \textbf{1} (2022), 1-18.

\bibitem[Duv73]{Duv73} G. Duvaut, \textit{R\'esolution d'un probl\`eme de Stefan (Fusion d'un bloc de glace a zero degr\'ees)}, C. R. Acad. Sci. Paris \textbf{276} (1973), 1461-1463.

    \bibitem[FS83]{FS83} E. Fabes, S. Salsa, \textit{Estimates of caloric measure and the initial-Dirichlet problem for the heat equation in Lipschitz cylinders}, Trans. Amer. Math. Soc. \textbf{279} (1983), 635-650.
    
    \bibitem[Fig18]{Fig18} A. Figalli, \textit{Regularity of interfaces in phase transitions via obstacle problems}, Prooceedings of the International Congress of Mathematicians (2018).   
    
    \bibitem[FS19]{FS19} A. Figalli, J. Serra, \textit{On the fine structure of the free boundary for the classical obstacle problem}, Invent. Math. \textbf{215} (2019), 311-366.
    
    \bibitem[FRS24]{FRS24} A. Figalli, X. Ros-Oton, J. Serra, \textit{The singular set in the Stefan problem}, J. Amer. Math. Soc. \textbf{37} (2024), 305-389.
    
    \bibitem[Fri68]{Fri68} A. Friedman, \textit{The Stefan problem in several space variables}, Trans. Amer. Math. Soc. \textbf{133} (1968), 51–87.
    
\bibitem[HS15]{HS15} M. Had\v{z}i\'c, S. Shkoller, \textit{Stability and decay in the classical Stefan problem}, Comm. Pure Appl. Math. \textbf{68} (2015), 689–757.

\bibitem[HR19]{HR19} M. Had\v{z}i\'c, P. Rapha\"el, \textit{On melting and freezing for the 2D radial Stefan problem}, J. Eur. Math. Soc. \textbf{21} (2019), 3259–3341.

\bibitem[HV97]{HV97} M. A. Herrero, J. J. Vel\'azquez, \textit{On the melting of ice balls}, SIAM J. Math. Anal. \textbf{28} (1997), 1–32.

\bibitem[Kim03]{Kim03} I. Kim, \textit{Uniqueness and Existence of Hele-Shaw and Stefan problem}, Arch. Rat. Mech. Anal. \textbf{168} (2003), 299–328.

    \bibitem[KN77]{KN77} D. Kinderlehrer, L. Nirenberg, \textit{Regularity in free boundary problems}, Ann. Sc. Norm. Sup. Pisa \textbf{4} (1977), 373-391.
    
    \bibitem[KN78]{KN78} D. Kinderlehrer, L. Nirenberg, \textit{The smoothness of the free boundary in the one-phase Stefan problem}, Comm. Pure Appl. Math. \textbf{31} (1978), 257–282.
    
    \bibitem[Koc98]{Koc98} H. Koch,  \textit{Classical solutions to phase transition problems are smooth}, Comm. Partial Differential Equations \textbf{23} (1998), 389–437.
    
    \bibitem[LC31]{LC31} G. Lam\'e, B. P. Clapeyron, \textit{M\'emoire sur la solidification par refroidissement d’un globe liquide}, Ann. Chimie Physique \textbf{47} (1831), 250-256.
    
\bibitem[PSS07]{PSS07} J. Pr\"uss, J. Saal, G. Simonett, \textit{Existence of analytic solutions for the classical Stefan problem}, Math. Ann. \textbf{338} (2007), 703–755.
    
    \bibitem[Ste91]{Ste91} J. Stefan, \textit{Ueber die Theorie der Eisbildung, insbesondere ueber die Eisbildung im Polarmeere}, Ann. Physik Chemie \textbf{42} (1891), 269-286.
    
    \bibitem[LM15]{LM15}
    E. Lindgren, R. Monneau, \textit{Pointwise regularity of the free boundary for the parabolic obstacle problem}, Calc. Var. Partial Differential Equations \textbf{54} (2015), 299-347.
    
    \bibitem[Tor24]{Tor24} C. Torres-Latorre, \textit{Parabolic boundary Harnack inequalities with right-hand side}, Arch. Rational Mech. Anal. \textbf{248} (2024), 73.
    
    \bibitem[Wei99]{Wei99} G. Weiss, \textit{Self-similar blow-up and Hausdorff dimension estimates for a class of parabolic free boundary problems}, SIAM J. Math. Anal. \textbf{30} (1999), 623-644.
    
\end{thebibliography}
\end{document}